\documentclass[a4paper,12pt,oneside]{article}

\usepackage{tikz}
\usetikzlibrary{shapes,arrows,positioning} 
\tikzset{
    >=stealth',
    punkt/.style={
           rectangle,
           rounded corners,
           draw=black, very thick, 
           text width=8em,
           minimum height=2em,
           text centered, fill=white, drop shadow},
    punkta/.style={
           rectangle,
           rounded corners,
           draw=black, very thick,
           text width=10em,
           minimum height=2em,
           text centered, fill=white, drop shadow},
    punktaka/.style={
           rectangle,
           rounded corners,
           draw=black, very thick,
           text width=14em,
           minimum height=2em,
           text centered, fill=white, drop shadow},       
    punktaa/.style={
           rectangle,
           rounded corners,
           draw=black, very thick,
           text width=15em,
           minimum height=2em,
           text centered, fill=white, drop shadow},
    punktaaa/.style={
           rectangle,
           rounded corners,
           draw=black, very thick,
           text width=10em,
           minimum height=2em,
           text centered, fill=white, drop shadow},
    pil/.style={
           ->,
           thick,
           shorten <=2pt,
           shorten >=2pt,}
}

\usepackage[framemethod=tikz]{mdframed}
\usepackage{soul}
\usepackage{transparent}
\usetikzlibrary{arrows, decorations.markings}
\usetikzlibrary{automata,positioning}

\usepackage[T1]{fontenc}
\usepackage{tikz}
\usetikzlibrary{shadings,shadows,shapes.arrows}

\usetikzlibrary{shapes,snakes}
\usetikzlibrary{matrix}

\usepackage{fancybox}

\usepackage{etex}
\usepackage{tcolorbox}
\usepackage{color}
\usepackage{fancybox,framed}

\usepackage[framemethod=tikz]{mdframed}
\usepackage{lipsum}
\usepackage{amssymb}
\usepackage{pifont}
\definecolor{mycolor}{rgb}{0.122, 0.435, 0.698}
\usepackage{emptypage}

\usepackage{afterpage}

\usepackage{epigraph}
\usepackage{url}
\usepackage{fancyhdr}

\usepackage{enumitem}
\usepackage{wrapfig}
\usepackage[toc,page]{appendix}

\usepackage{tikz}
\usepackage{fp}
\usepackage{pgf}
\usepackage{xifthen}
\newcommand{\warsawApp}[2] 
{
\FPeval{\points}{4-((#1)/2)}
\begin{tikzpicture}[scale=5,domain=0:1] 
\tikzstyle{every node}=[circle, draw, fill=black!50,
                        inner sep=0pt, minimum width=\points pt]
\FPeval{\step}{1/2^((#1))} 
\FPeval{\twostep}{1/2^(2*((#1)))}
\FPeval{\twostepone}{1/2^(2*((#1))-1)}
\FPeval{\twosteptwo}{1/2^(2*((#1))-2)}
\FPeval{\stepone}{1/2^(((#1))-1)}
\FPeval{\B}{(1/2)+(1/2^(((#1))-1))}
\FPeval{\A}{(1/2)-(1/2^(((#1))-1))}
\FPeval{\C}{(1/2)+(1/2^(((#1))))}
\FPeval{\D}{1-(1/2^((#1)))}
\FPeval{\E}{1/2+(3/2^((#1)))}
\FPeval{\kminustwo}{((#1))-1}
\pgfmathsetmacro{\puntosa}{int((((#1))-1)/2)}
\pgfmathsetmacro{\puntosb}{int((((#1))-2)/2)}
\ifthenelse{#2=1}{
								\draw[step=\stepone,gray,very thin] (-0.1,-0.1) grid (1.1,1.1);
							  }{}	
\draw (0,1)--(0,0)--(1,0)--(1,1)--(1/2,1)--(1/2,1/2)--(1/4,1/2)--(1/4,1)--(1/8,1)--
(1/8,1/2)--(1/16,1/2)--(1/16,1)--(1/32,1)--(1/32,1/2)--(1/64,1/2)--(1/64,1);
\draw
\foreach \x in {0,\stepone,...,1} {(\x,0) node {} }

\foreach \x in {0,\stepone,...,\A} {(0,\x) node {} }
\foreach \x in {0,\stepone,...,\A} {(1,\x) node {} }

\foreach \l in {0,1,...,\kminustwo} {\foreach \m in {0.5,\B,...,1} {(1/2^\l,\m) node {} }}
\foreach \m in {0.5,\B,...,1} {(0,\m) node {} };

\foreach \x in {1,...,\puntosa} {
													\pgfmathsetmacro{\u}{1/2^(2*\x-1)}
													\pgfmathsetmacro{\v}{1/2^(2*\x-1)+\stepone}
													\pgfmathsetmacro{\w}{1/2^(2*\x-2)}
													\foreach \z in {\u,\v,...,\w} {
																								\draw node at (\z,1) {} ;
																								}
										          }
										          
\ifthenelse{#1 > 3}{
                             \foreach \x in {1,...,\puntosb} {
													                          \pgfmathsetmacro{\u}{1/2^(2*\x)}
													                          \pgfmathsetmacro{\v}{1/2^(2*\x)+\stepone}
													                          \pgfmathsetmacro{\w}{1/2^(2*\x-1)}
													                          \foreach \z in {\u,\v,...,\w} {
																								                          \draw node at (\z,0.5) {} ;
																								                         }
										                                     }
							   }{}			                               			
\ifthenelse{#2=0}{
							\foreach \m in {\C,\E,...,\D} {\draw node at (\step,\m) {} ;}	
							}{}
\end{tikzpicture}}
\newcommand{\warsawPol}[2] 
{
\FPeval{\points}{4-((#1)/2)}
\begin{tikzpicture}[scale=#2,domain=0:1] 
\tikzstyle{every node}=[circle, draw, fill=black!50,
                        inner sep=0pt, minimum width=\points pt]
\FPeval{\step}{1/2^((#1))} 
\FPeval{\twostep}{1/2^(2*((#1)))}
\FPeval{\twostepone}{1/2^(2*((#1))-1)}
\FPeval{\twosteptwo}{1/2^(2*((#1))-2)}
\FPeval{\stepone}{1/2^(((#1))-1)}
\FPeval{\steptwo}{1/2^(((#1))-2)}
\FPeval{\B}{(1/2)+(1/2^(((#1))-1))}
\FPeval{\BB}{(1/2)+(1/2^(((#1))-2))}
\FPeval{\A}{(1/2)-(1/2^(((#1))-1))}
\FPeval{\C}{(1/2)+(1/2^(((#1))))}
\FPeval{\D}{1-(1/2^((#1)))}
\FPeval{\E}{1/2+(3/2^((#1)))}
\FPeval{\F}{1-(1/2^(((#1))-1))}
\FPeval{\kminustwo}{((#1))-1}
\pgfmathsetmacro{\puntosa}{int((((#1))-1)/2)}
\pgfmathsetmacro{\puntosb}{int((((#1))-2)/2)}
\pgfmathsetmacro{\puntosc}{int((((#1))+1)/2)}
\draw (0,1)--(0,0)--(1,0)--(1,1);
\foreach \x in {1,...,\puntosa}{\pgfmathsetmacro{\u}{1/2^(2*\x-1)}
											     \pgfmathsetmacro{\w}{1/2^(2*\x-2)}
												  \draw (\u,1)--(\w,1);
											   }
\foreach \x in {1,...,\puntosc}{\pgfmathsetmacro{\u}{1/2^\x}										   
												  \draw (\u,1)--(\u,0.5);
											   }
\ifthenelse{#1 > 3}{
\foreach \x in {1,...,\puntosb} {\pgfmathsetmacro{\u}{1/2^(2*\x)}
													\pgfmathsetmacro{\w}{1/2^(2*\x-1)}
													\draw (\u,0.5)--(\w,0.5);
										        }	
							    }{}

\draw (\steptwo,0.5)--(\steptwo,1);	
\draw (\stepone,0.5)--(\stepone,1);
\foreach \x in {0.5,\B,...,1}{\draw (\stepone,\x)--(\steptwo,\x);}	
\draw (0,0.5)--(\stepone,0.5);
\draw (0,1)--(\stepone,1);		
\draw (\step,\C)--(\step,\D);
\ifthenelse{#1 > 3}{
\foreach \m in {\C,\E,...,\F}  {
													\pgfmathsetmacro{\mstepone}{\m+\stepone}
													\pgfmathsetmacro{\mstep}{\m+\step}
													
													\draw [fill=gray, fill opacity=0.85] (\step,\m)--(0,\mstep)--
													(\step,\mstepone)--(\stepone,\mstep)--(\step,\m);																					
									          }	
							   }{
							     \pgfmathsetmacro{\mstepone}{\C+\stepone}
								 \pgfmathsetmacro{\mstep}{\C+\step}
													
									\draw [fill=gray, fill opacity=0.85] (\step,\C)--(0,\mstep)--
													(\step,\mstepone)--(\stepone,\mstep)--(\step,\C);			
								}		
\draw (0,0.5)--(\step,\C)--(\stepone,0.5);	
\draw (0,1)--(\step,\D)--(\stepone,1);
\foreach \x in {\B,\BB,...,\F}{\draw[thin, densely dotted] (0,\x)--(\stepone,\x);}
	 							          	
\draw [fill=gray, fill opacity=0.5]
       (0,0.5) -- (\stepone,0.5) -- (\stepone,1) -- (0,1)--(0,0.5);					       											
\draw
\foreach \x in {0,\stepone,...,1} {(\x,0) node {} }

\foreach \x in {0,\stepone,...,\A} {(0,\x) node {} }
\foreach \x in {0,\stepone,...,\A} {(1,\x) node {} }

\foreach \l in {0,1,...,\kminustwo} {\foreach \m in {0.5,\B,...,1} {(1/2^\l,\m) node {} }}
\foreach \m in {0.5,\B,...,1} {(0,\m) node {} };

\foreach \x in {1,...,\puntosa} {
													\pgfmathsetmacro{\u}{1/2^(2*\x-1)}
													\pgfmathsetmacro{\v}{1/2^(2*\x-1)+\stepone}
													\pgfmathsetmacro{\w}{1/2^(2*\x-2)}
													\foreach \z in {\u,\v,...,\w} {
																								\draw node at (\z,1) {} ;
																								}
										          }
\ifthenelse{#1 > 3}{										          
\foreach \x in {1,...,\puntosb} {
													\pgfmathsetmacro{\u}{1/2^(2*\x)}
													\pgfmathsetmacro{\v}{1/2^(2*\x)+\stepone}
													\pgfmathsetmacro{\w}{1/2^(2*\x-1)}
													\foreach \z in {\u,\v,...,\w} {
																								\draw node at (\z,0.5) {} ;
																								}
										          }		
								}{}	          	
\foreach \m in {\C,\E,...,\D} {\draw node at (\step,\m) {} ;}				          
\end{tikzpicture}}
\newcommand{\warsawTdPol}[2]
{
\FPeval{\points}{4-((#1)/2)}
\begin{tikzpicture}[scale=#2,line join=bevel,x=5,y=5,z=3,rotate =0]
\tikzstyle{every node}=[circle, draw, fill=black!50,
                        inner sep=0pt, minimum width=\points pt]                        
\FPeval{\step}{1/2^((#1))} 
\FPeval{\twostep}{1/2^(2*((#1)))}
\FPeval{\twostepone}{1/2^(2*((#1))-1)}
\FPeval{\twosteptwo}{1/2^(2*((#1))-2)}
\FPeval{\stepone}{1/2^(((#1))-1)}
\FPeval{\steptwo}{1/2^(((#1))-2)}
\FPeval{\B}{(1/2)+(1/2^(((#1))-1))}
\FPeval{\BB}{(1/2)+(1/2^(((#1))-2))}
\FPeval{\A}{(1/2)-(1/2^(((#1))-1))}
\FPeval{\C}{(1/2)+(1/2^(((#1))))}
\FPeval{\D}{1-(1/2^((#1)))}
\FPeval{\H}{1-(1/2^((#1)))-\step}
\FPeval{\E}{1/2+(3/2^((#1)))}
\FPeval{\F}{1-(1/2^(((#1))-1))}
\FPeval{\kminustwo}{((#1))-1}
\pgfmathsetmacro{\puntosa}{int((((#1))-1)/2)}
\pgfmathsetmacro{\puntosb}{int((((#1))-2)/2)}
\pgfmathsetmacro{\puntosc}{int((((#1))+1)/2)}
\FPeval{\alt}{\stepone}
\draw (0,0,1)--(0,0,0)--(1,0,0)--(1,0,1);
\foreach \x in {1,...,\puntosa}{\pgfmathsetmacro{\u}{1/2^(2*\x-1)}
											     \pgfmathsetmacro{\w}{1/2^(2*\x-2)}
												  \draw (\u,0,1)--(\w,0,1);
											   }
\foreach \x in {1,...,\puntosc}{\pgfmathsetmacro{\u}{1/2^\x}										   
												  \draw (\u,0,1)--(\u,0,0.5);
											   }
\ifthenelse{#1 > 3}{
\foreach \x in {1,...,\puntosb} {\pgfmathsetmacro{\u}{1/2^(2*\x)}
													\pgfmathsetmacro{\w}{1/2^(2*\x-1)}
													\draw (\u,0,0.5)--(\w,0,0.5);
										        }	
								}{}		        
\draw (\steptwo,0,0.5)--(\steptwo,0,1);	
\draw (\stepone,0,0.5)--(\stepone,0,1);
\foreach \x in {0.5,\B,...,1}{\draw (\stepone,0,\x)--(\steptwo,0,\x);}	
\draw (0,0,0.5)--(\stepone,0,0.5);
\draw (0,0,1)--(\stepone,0,1);		
\draw
\foreach \x in {0,\stepone,...,1} {(\x,0,0) node {} }

\foreach \x in {0,\stepone,...,\A} {(0,0,\x) node {} }
\foreach \x in {0,\stepone,...,\A} {(1,0,\x) node {} }

\foreach \l in {0,1,...,\kminustwo} {\foreach \m in {0.5,\B,...,1} {(1/2^\l,0,\m) node {} }}
\foreach \m in {0.5,\B,...,1} {(0,0,\m) node {} };

\foreach \x in {1,...,\puntosa} {
													\pgfmathsetmacro{\u}{1/2^(2*\x-1)}
													\pgfmathsetmacro{\v}{1/2^(2*\x-1)+\stepone}
													\pgfmathsetmacro{\w}{1/2^(2*\x-2)}
													\foreach \z in {\u,\v,...,\w} {
																								\draw node at (\z,0,1) {} ;
																								}
										          }		
\ifthenelse{#1 > 3}{										          							          
 \foreach \x in {1,...,\puntosb} {
													\pgfmathsetmacro{\u}{1/2^(2*\x)}
													\pgfmathsetmacro{\v}{1/2^(2*\x)+\stepone}
													\pgfmathsetmacro{\w}{1/2^(2*\x-1)}
													\foreach \z in {\u,\v,...,\w} {
																								\draw node at (\z,0,0.5) {} ;
																								}
										          }		
							  }{}		   											          		
\ifthenelse{#1 > 3}{
	\foreach \m in {\C,\E,...,\H} {            								
											  \pgfmathsetmacro{\mpstepone}{\m+\stepone}
											  \pgfmathsetmacro{\mpstep}{\m+\step}
											  \pgfmathsetmacro{\mmstep}{\m-\step}
											  \coordinate (ul) at (0,0,\mpstep);	
											  \coordinate (ur) at (\stepone,0,\mpstep);	
											  \coordinate (dl) at (0,0,\mmstep);		
											  \coordinate (dr) at (\stepone,0,\mmstep);
											  \draw [fill=gray, fill opacity=1] (\step,\alt,\m)--(ul)--(ur)--cycle;						 
						    				  }	
						       }{\draw [fill=gray, fill opacity=1] (\step,\alt,\C)--(0,0,\B)--(\stepone,0,\B)--cycle;}
\draw node at (0,0,\B) {};						       
\ifthenelse{#1 > 3}{							    				  
\foreach \m in {\C,\E,...,\H} {            								
											  \pgfmathsetmacro{\mpstepone}{\m+\stepone}
											  \pgfmathsetmacro{\mpstep}{\m+\step}
											  \pgfmathsetmacro{\mmstep}{\m-\step}
											  \coordinate (ul) at (0,0,\mpstep);	
											  \coordinate (ur) at (\stepone,0,\mpstep);	
											  \coordinate (dl) at (0,0,\mmstep);		
											  \coordinate (dr) at (\stepone,0,\mmstep);
											  \draw [fill=gray, fill opacity=1] (\step,\alt,\m)--(ur)--(\step,\alt,\mpstepone)--cycle;
								            }	
								 }{\draw [fill=gray, fill opacity=1] (\step,\alt,\C)--(\stepone,0,\B)--(\step,\alt,\E)--cycle;}						            		
\foreach \m in {\C,\E,...,\D} {            								
											  \pgfmathsetmacro{\mpstep}{\m+\step}
											  \pgfmathsetmacro{\mmstep}{\m-\step}
											  \coordinate (ul) at (0,0,\mpstep);	
											  \coordinate (ur) at (\stepone,0,\mpstep);	
											  \coordinate (dl) at (0,0,\mmstep);		
											  \coordinate (dr) at (\stepone,0,\mmstep);							 
											  \draw [fill=gray, fill opacity=0.5] (\step,\alt,\m)--(ur)--(dr)--cycle;
										    }								 
\draw [fill=gray, fill opacity=0.5] (\step,\alt,\C)--(0,0,0.5)--(\stepone,0,0.5)--cycle;			
\foreach \m in {0.5,\B,...,1} {\draw node at (\stepone,0,\m) {} ;}
\foreach \m in {\C,\E,...,\D} {\draw node at (\step,\alt,\m) {} ;}
\draw node at (0,0,0.5) {};
\end{tikzpicture}
}

\usepackage[utf8]{inputenc}
\usepackage{color}
\usepackage{import}
\usepackage{setspace}
\usepackage{latexsym,amsmath,amsfonts,amscd,amssymb,amsthm}
\usepackage{graphicx}
\usepackage[all]{xy}
\usepackage{epstopdf}
\DeclareGraphicsExtensions{.pdf,.png,.jpg,.gif, .eps}

\usepackage{amsthm}
\usepackage{thmtools}

\theoremstyle{plain} 
\newtheorem{teo}{Theorem}
\newtheorem{lem}{Lemma}
\newtheorem{prop}{Proposition}
\newtheorem{cor}{Corollary}
\theoremstyle{definition}

\newtheorem{cons}{Construction}

\theoremstyle{remark}
\newtheorem{obs}{Remark}
\newtheorem{ej}{Example}

\newcommand{\comillas}[1]{``#1''}

\newcommand{\conjunto}[1]{\left\lbrace #1 \right\rbrace}
\newcommand{\mapeo}[5]{
\begin{eqnarray*}
#1:#2 & \longrightarrow & #3\\
#4 & \longmapsto & #5
\end{eqnarray*}}
\newcommand{\cms}{$(X,\textrm{d})$ }
\newcommand{\U}[1]{U_{2\varepsilon_{#1}}(A_{#1})}
\newcommand{\todon}{n\in\mathbb{N}}

\renewcommand{\epsilon}{\varepsilon}
\newcommand{\ball}[2]{\textrm{B}(#1,#2)}
\newcommand{\dist}[2]{\textrm{d}(#1,#2)}
\newcommand{\fas}{\{\varepsilon_n,A_n,\gamma_n,\}_{\todon}}
\newcommand{\card}{\textrm{card}}
\newcommand{\diam}{\textrm{diam}}

\usepackage[light,math]{iwona}
\makeindex
\makeatletter
\newcommand{\subjclass}[2][2010]{%
  \let\@oldtitle\@title%
  \gdef\@title{\@oldtitle\footnotetext{#1 \emph{Mathematics Subject Classification.} #2}}%
}
\newcommand{\keywords}[1]{%
  \let\@@oldtitle\@title%
  \gdef\@title{\@@oldtitle\footnotetext{\emph{Key words and phrases.} #1.}}%
}
\makeatother
\title{Reconstruction of compacta by finite approximations and Inverse Persistence}
\usepackage{authblk}
\author[1]{Diego Mondéjar Ruiz}
\author[2]{Manuel A. Morón}
\affil[1]{Departamento de Matemática Aplicada y Estadística\protect\\Universidad San Pablo CEU, Madrid, Spain\protect\\ \textsf{diego.mondejarruiz@ceu.es}}
\affil[2]{Departamento de Álgebra,Geometría y Topología\protect\\Universidad Complutense de Madrid and\protect\\Instituto de Matematica Interdisciplinar, Madrid, Spain\protect\\ \textsf{mamoron@mat.ucm.es}}
\date{}                     
\begin{document}
\maketitle
\begin{abstract}
The aim of this paper is to show how the homotopy type of compact metric spaces can be reconstructed by the inverse limit of an inverse sequence of finite approximations of the corresponding space. This recovering allows us to define inverse persistence as a new kind of persistence process.
\end{abstract}
\section{Introduction}
In this paper we deal with approximations of metric compacta. The approximation and reconstruction of topological spaces using simpler ones is an old theme in geometric topology. One would like to construct a very simple space as similar as possible to the original space. Since it is very difficult (or does not make sense) to obtain a homeomorphic copy, the goal will be to find an space preserving some (algebraic) topological properties such as compactness, connectedness, separation axioms, homotopy type, homotopy and homology groups, etc.
\paragraph{}
The first candidates to act as the simple spaces reproducing some properties of the original space are polyhedra. See the survey \cite{Mapproximating} for the main results. In the very beginnings of this idea, we must recall the studies of Alexandroff around 1920, relating the dimension of compact metric spaces with dimension of polyhedra by means of maps with controlled (in terms of distance) images or preimages. In a more modern framework, the idea of approximation can be carried out constructing a simplicial complex, based on our space, such as the Vietoris-Rips complex or the \v{C}ech complex, and compare its realization with it. In this direction, for example, we find the classical Nerve Lemma \cite{Bon,dRcomplexes} which claims that for a ``good enough" open cover of the space (meaning an open covering with contractible or empty members and intersections), the nerve of the cover has the homotopy type of our original space. The problem is to find those good covers (if they exist). For Riemannian manifolds, there are some results concerning its approximation by means of the Vietoris-Rips complex. Hausmann showed \cite{Hon} that the realization of the Vietoris-Rips complex of the manifold, for a small enough parameter choice, has the homotopy type of the manifold. In \cite{Lvietoris}, Latschev proved a conjecture made by Hausmann: The homotopy type of the manifold can be recovered using only a (dense enough) finite set of points of it, for the Vietoris-Rips complex. The results of Petersen \cite{Pa}, comparing the Gromov-Hausdorff distance of metric compacta with their homotopy types, are also interesting. Here, polyhedra are just used in the proofs, not in the results. 
\paragraph{}
Another important point, concerning this topic, are finite topological spaces. It could be expected that finite topological spaces are too simple to capture any topological property, but this is far from reality and comes from thinking about finite spaces as discrete ones. Another obstruction to the use of finite spaces is that a very basic observation reveals that they have very poor separation properties. Any finite topological space satisfying just the $T_1$ axiom of separation is really a discrete space. Since non-Hausdorff spaces seem to be less manageable, finite spaces could represent themselves a more difficult problem to study that the spaces we want to approximate with. There were two papers of Stong \cite{Sfinite} and McCord \cite{Msingular} that were a breakthrough in finite topological spaces. Stong studied the homeomorphism and homotopy type of finite spaces. Among other results, he showed that the homeomorphism types are in bijective correspondence with certain equivalence classes of matrices and that every finite space has a \emph{core}, which is homotopy equivalent to it. McCord defined a functor from finite $T_0$ spaces to polyhedra preserving the homotopy and homology groups (defining a weak homotopy equivalence between them). This is a very important result, since we obtain that the homotopy and homology groups of every compact polyhedron can be obtained as the groups of a finite space. The essential property of finite spaces, making possible this result, is that the arbitrary intersection of open sets is open (every space satisfying this property is called Alexandroff space) and hence they have minimal basis. If the finite space is $T_0$, the minimal basis gives the space a structure of a poset (first noticed in \cite{Adiskrete}) which is used in the cited result. Both papers were retrieved in a series of very instructive notes by May \cite{Mfinitetopological,Mfinitecomplexes}, where these results are adequately valued. Based on the theorems and relations proved in those papers, Barmak and Minian \cite{BMsimple,BMautomorphism,Balgebraic} introduced a whole algebraic topology theory over finite spaces.
\paragraph{}
One step further is to make use of the inverse limit construction. If we cannot obtain the desired approximation using only one simple space, we can try to obtain it as some kind of limit of an infinite process of refinement by good spaces. That idea is accomplished by the notion of inverse limit. It is similar in spirit to the use of the Taylor series to approximate a function. For the origins of using inverse limits to approximate compacta, we should go back, again, to the work of Alexandroff \cite{Auntur}, where it is shown that every compact metric space has an associate inverse sequence of finite $T_0$ spaces such that there is a subspace of the inverse limit homeomorphic to the original one. We also have to mention Freudenthal, who showed \cite{Freudenthal} that every compact metric space is the inverse limit of an inverse sequence of polyhedra. More recent results were obtained by Kopperman et al \cite{KTWthe,KWfinite}. They showed that every compact Hausdorff space is the \emph{Hausdorff reflection} of the inverse limit of an inverse sequence of finite $T_0$ spaces.  Also they define the concept of \emph{calming map} and show that if the maps in this sequence are calming, then an inverse sequence of polyhedra can be associated and its limit is homeomorphic to the original space. Those are good results, although the technical concepts of Hausdorff reflection and calming map, make its real computation hard to achieve. Another important result is the one obtained by Clader \cite{Cinverse}, who proved that every compact polyhedron has the homotopy type of the inverse limit of an inverse sequence of $T_0$ finite spaces.
\paragraph{}
Shape theory makes use of this notion of approximation by inverse limits. This theory was founded in 1968 with Borsuk's paper \cite{Bconcerning}. It is a theory developed to extend homotopy theory for spaces where it does not work well, because of its pathologies (for example, bad local properties). Although Borsuk's original approach does not make explicit use of inverse limits, they are in the underlying machinery. The idea of Borsuk was to enlarge the set of morphisms between metric compacta by embedding the spaces into the Hilbert cube and define some kind of morphisms between the open neighborhoods of those embedded spaces. Later, Mardesic and Segal initiated in \cite{MSshapes} the inverse system approach to Shape theory. Here, the approximative sense of Shape theory is clear: Every compact Hausdorff space can be written as the inverse limit of an inverse system (or an inverse sequence if the space is metric) of compact ANR's, which act as the good spaces. Then, the new morphisms are essentially defined as maps between the systems. Shape theory, in its invese system approach, is then defined and developed \cite{MSshape} for more general topological spaces and new concepts, as \emph{expansions} and \emph{resolutions}, have to take the role of the inverse limit for technical reasons, but the point of view is similar. It is evident that the inverse limit approximation point of view for spaces is closely related with Shape theory. There are several shape invariants. Among others, we have the \v{C}ech homology, which is the inverse limit of the singular homology groups and the induced maps in homology of the inverse system defining the shape of the space. 
\paragraph{}
In the last years, there has been a renewed interest in the approximation and reconstruction of topological spaces, in part because the development of the Computational Topology and more concretely the Topological Data Analysis (read the excellent survey of Carlsson \cite{Ctopology} as an introduction for this topic). Here, the idea is to recapture the topological properties of some space using partial or defective (sometimes called noisy) information about it. Usually we only know a finite set of points and the distances between them (this is known as \emph{point cloud}) which is a sample of an unknown topological space, and the goal is to reconstruct the topology of the space or, at least, be able to detect some topological properties. Besides the classical Vietoris-Rips and \v{C}ech complexes, several other complexes (as the \emph{witness}, \emph{Delaunay} complexes or the \emph{alpha shapes} \cite{EHcomputational}) are defined with this purpose. Some important results in this setting were obtained by Niyogi et al \cite{NSWa,NSWfinding}, where they give conditions to reconstruct the homotopy type and the homology of the manifold when only a finite set of points (possibly with noise) lying in a submanifold of some euclidean space, is known. They also use probability distributions in their results. There are a large amount of recent papers devoted to this kind of reconstructions. For instance, Attali et al \cite{Avietoris}, in a more computational approach, give conditions in which a Vietoris-Rips complex of a point cloud in an euclidean space recovers the homotopy type of the sampled space. Among other techniques, we have to highlight the \emph{persistent homology}. The idea here is as easy as effective: Instead of considering only one polyhedron based on the point cloud to recover the topology of the hidden space, consider a family of polyhedra constructed from the data and natural maps induced by the inclusion connecting them. Then, we do not choose one concrete resolution to analyze the point cloud, but we consider all possible values of the parameter and their connections at once and use them together to determine the evolution of the topology of the point cloud along the parameter changes. 
\paragraph{}
The first link between Shape theory and Persistent Homology was made in 1999 by Vanessa Robins \cite{Rtowards}. There, she proposes to use the machinery of shape theory to approximate compact metric spaces from finite data sets. She introduced the concept of \emph{persistent Betti number}, which is the evolution of the Betti numbers in the inverse sequence of polyhedra at different scales (or resolution) of approximation. Her approach is the following: Given a sample (finite set of points, possibly with noise) of an unknown topological space, construct an inverse system of $\varepsilon$-neighborhoods of the finite set and inclusion maps. Then, triangulate the $\varepsilon$-neighborhoods using the $\alpha$-shapes and we obtain an inverse system of polyhedra based on the sample. Then, track the Betti numbers over this system. For some examples arising from dynamical systems, she is able to give bounds for the behavior of the Betti numbers, when the resolution parameter tends to infinity, and hence the sample is more accurate. Her guess is that the more accurate the sample is, the more exactness in the prediction can be made, and is here where shape theory is proposed as a theory to support this and other similar methods.
\paragraph{}
In 2008, following this direction, Mor\'{o}n et al \cite{MCLepsilon} introduced what they called the \emph{main construction}\footnote{This was not the first paper of this research group in this topic. From another point of view, this theme is treated in \cite{GMRSfinite}.}. This is an inverse sequence of finite topological spaces constructed from more and more tight approximations of a given compact metric space. The finite spaces are not exactly the approximations but some subspaces of the hyperspace of the approximations with the upper semifite topology. This is necessary in order to define continuous maps between these approximations. These maps are defined in terms of proximity between points of consecutive approximations. Hence, they are not the inclusion (because the finite spaces are not necesarilly nested). At this point, they make use of the so called \emph{Alexandroff-McCord correspondence}, which is the functor assigning a polyhedron to every $T_0$ finite space, mentioned above. The functoriality is used to define maps between the induced polyhedra and hence we obtain an inverse sequence of polyhedra. The way that this sequence is constructed, using finite approximations, induces them to conjeture that the inverse limit of the inverse sequence of polyhedra is somehow related with the topology of the original compact metric space. This conjecture is stated as the \emph{general principle}, proposing this sequence to detect the shape properties of the space such as the \v{C}ech homology. Our work is placed here, understanding and expanding the properties of the main construction. In a forthcoming paper \cite{preprintShape} it is shown that this inverse sequence represents this space in terms of shape, namely that it is an HPol-expansion of the original space, and hence it recovers all the shape information about it.
\paragraph{}
We show here that the inverse limit of every inverse sequence of finite spaces defined by the main construction has the homotopy type of the original space, and it contains an homeomorphic copy as a (strong) deformation retract. We identify explicitly this subspace. After that, we study some properties of the main construction and the result of performing the main construction to some specific classes of spaces as dense subspaces, countable and ultrametric spaces. For the last, we obtain that in this case we can choose a suitable construction such that the inverse limit of the finite spaces is homeomorphic to the ultrametric space. We compare our results with that of Clader and Kopperman et al (previously cited). We can obtain Clader's result as a corollary of our main theorem. In the other case, their approximations are made for Hausdorff compact spaces, and we do not obtain their generality. In contrast, for the case of metric compacta, we obtain the same consecuences, and we also deduce that every compact metric space has the homotopy type of the inverse limit of an inverse sequence of finite $T_0$ spaces. The same result has been generalized very recently by Bilski \cite{Bilskiinverse} in a more general setting, namely for the class of locally compact, paracompact, Hausdorff spaces, using inverse systems of $T_0$ Alexandroff spaces instead of inverse sequences of finite $T_0$ spaces. The inverse system is defined in terms of the partial order defined in the set of all locally finite open coverings of the space and is related in some sense to the construction in \cite{KWfinite}. As a consecuence, the results of Clader \cite{Cinverse} are generalized using barycentric subdivisions of Alexandroff spaces. Also, we show that the Hausdorff reflection preserves the shape type and hence the results of Kopperman et al implies that every Hausdorff compact space has the same shape as an inverse sequence of finite $T_0$ spaces. The main construction allows us to outline an algorithm to obtain persistence modules as finite sequences extracted from an inverse sequence of polyhedra. These persistence modules are obtained in a different way from the usual ones so we call this new point of view \emph{inverse persistence}. We outline the basics of this process in section \ref{inversepersistence} and perform it in our paradigmatic example, the Warsaw circle. The constructibility of our process makes it suitable for computational purposes. All the results shown here are part of the first authors' thesis dissertation \cite{phMondejar}.
\paragraph{} The rest of this section is devoted to present the basic elements needed in the paper.
 
\subsection{Shape Theory}\label{shapetheory}
Shape theory is a suitable extension of homotopy theory for topological spaces with bad local properties, where this theory does not give any information about the space. The paradigmatic example is the Warsaw circle $\mathcal{W}$: It is the graph of the function 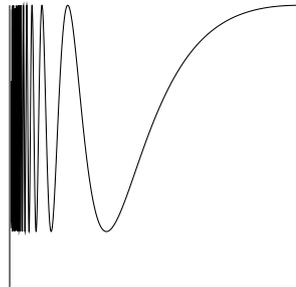
\begin{wrapfigure}{r}{0.5\textwidth}
\begin{center}
\begin{tikzpicture}[scale=1.5,domain=0:1,x=4cm]
\FPeval{\w}{2/(3.1415)}
\draw[domain=0.004:\w,samples=5000] plot (\x, {sin((1/\x)r)});
\draw (0,1)--(0,-1.5)--(\w,-1.5)--(\w,{sin((1/\w)r)});
\end{tikzpicture}
\end{center}
  \caption{The Warsaw circle.}
  \label{warsawcircle}
\end{wrapfigure} $\sin\left(\frac{1}{x}\right)$ in the interval $(0,\frac{2}{\pi}]$ adding its closure (that is, the segment joining $(0,-1)$ and $(0,1)$) and closing the space by any simple (not intersecting itself or the rest of the space) arc joining the points $(0,-1)$ and $(\frac{\pi}{2},1)$. See figure \ref{warsawcircle}. It is readily seen that the fundamental group of $\mathcal{W}$ is trivial. Moreover, so are all its homology and homotopy groups. But it is also easy to see that $\mathcal{W}$ has not the homotopy type of a point (for example, it decomposes the plane in two connected components), so it has some homotopy type information that the homotopy and homology groups are not able to capture. It is then evident that homotopy theory does not work well for $\mathcal{W}$. Shape theory was initiated by Karol Borsuk in 1968 to overcome these limitations, defining a new category, containing the same information about well behaved topological spaces, but giving some information about spaces with bad local properties. The idea is that, no matter how bad the space is, its neighborhoods when it is embedded into a larger space (for example the Hilbert cube $Q$) are not too bad. In our example, it is easy to see that the neighborhoods of $\mathcal{W}$ are annuli, having then the homotopy type of $\mathbb{S}^1$. The space $\mathcal{W}$ share some global properties with $\mathbb{S}^1$. There are no non-trivial maps from $\mathbb{S}^1$ to $\mathcal{W}$, so the method  will be to compare them in terms of maps between its neighborhoods. 

Specifically, Borsuk defined a new class of morphism between metric compacta embedded in the Hilbert cube, called fundamental sequences, as sequences of continuous maps $f_n:Q\rightarrow Q$ satisfying some homotopy conditions on the neighborhoods of the spaces embedded in the Hilbert cube.  He introduced a notion of homotopy among fundamental sequences, setting the shape category of metric compacta as the homotopy classes for this homotopy relation. It is shown that the new category differs only formally from the homotopy category when the space under consideration is an ANR. For the details, see the original source \cite{Bconcerning}, or the books \cite{Btheory,BtheoryA}. After Borsuk's description of the shape category for metric compacta, there was a lot of work in shape theory, such as different descriptions of shape, extensions to more general spaces (for instance, Fox's extension of shape for metric spaces \cite{Fon}), classifications of shape types or shape invariants. As general references, we recommend the books \cite{Btheory,BtheoryA,MSshape,DSshape} and the surveys \cite{Mthirty,Mabsolute}. The inverse system approach is the most widely used and it will be the one used here. It was initiated by Mardesic and Segal for compact Hausdorff spaces in \cite{MSshapes}, and it was developed by them and some other authors for more general situations. The best reference for this approach, is the book by the same authors \cite{MSshape} and all the proofs omitted here can be found there.

An \emph{inverse sequence} of topological spaces is a countable set of spaces $\{X_n\}_{n\in\mathbb{N}}$ and continuous maps $p_n:X_n\rightarrow X_{n-1}$ for $n\in\mathbb{N}$. We denote it by $\{X_n,p_{n,n+1}\}$ or
\[X_1\xleftarrow{p_{1,2}}X_2\xleftarrow{p_{2,3}}\ldots\xleftarrow{p_{n-1,n}} X_n\xleftarrow{p_{n,n+1}}X_{n+1}\xleftarrow{p_{n+1,n+2}}\ldots\]
The \emph{inverse limit} of an inverse sequence is the subset of the product space $\mathcal{X}\subset\prod_{n\in\mathbb{N}}X_n$ consisting of the points $(x_1,x_2,\ldots,x_n,x_{n+1},\ldots)$ satisfying $p_{n,n+1}(x_{n+1})=x_n$ for every $n\in\mathbb{N}$.

Given a compact metric space $X$, we consider an inverse sequence of compact ANRs (or polyhedra) $\{X_n,p_n\}$ with $X$ as inverse limit (it always exists). This inverse sequences up to a notion of equivalence are the new objetcs of our category. To define morphisms, we consider a map between inverse sequences and, again, a notion of equivalence\footnote{We do not include these technical definitions here for the sake of simplicity.}. Thus we have a new category where essentially we substitute compat metric spaces by their associated inverse sequences. This new category is able to detect some non trivial topological properties that homotopy is not. For example, the homology groups of the warsaw circle described above are trivial, but the equivalent groups in the shape category, named the \v{C}ech homology groups are not because they are detected in some way by the inverse sequence representing the warsaw circle. The idea is that we approximate spaces with a poor local behaviour with sequences of \comillas{nice} spaces. In our example, the inverse sequence is a sequence of circles an the identity map and its first \v{C}ech homology group is isomorphic to $\mathbb{Z}$ because it is the inverse limit of the induced inverse sequence of the first homology groups of the inverse sequence representing the warsaw circle. In the shape category, we have equivalent generalizations of homotopy and homology groups. It also have its own invariants such as the movability. This theory can be defined with more complex machinery for every topological space. Hence, we have an extension of the homotopy category, enlarging the set of morphisms. Not every shape morphism is represented by a continuous function, but we have that every continuous function induces a shape morphism. From \cite{Mshapes}, we have the following useful characterization for a function to induce an isomorphism in the shape category.
\begin{teo}\label{teo:funindiso}
Let $X$ and $Y$ be topological spaces and $f:X\rightarrow Y$ a continuous map. Then $f$ is a shape equivalence (that is, the shape morphism induced by $f$ is an isomorphism in the shape category) if and only if, for every CW-complex (or equivalently ANR or polyhedron) $P$, the function\footnote{Notation: For topological spaces $Z, R$, $[Z,R]$ is the set of homotopy classes of continuous functions from $Z$ to $R$. For a map $h:Z\rightarrow R$, we represent by $[h]$ its homotopy class.} \mapeo{f}{\left[ Y,P\right] }{\left[ X,P\right] }{\left[ h\right] }{\left[ h\cdot f\right] } is a bijection.
\end{teo}
		
There is another approach to shape that we shall use here. This is the multivalued theory of shape for metric compacta, initiated by Sanjurjo in \cite{San}. The key and acute idea of multivalued shape theory is to replace the shape morphisms by sequences of multivalued maps with decreasing diameters of their images, which is, in some sense, a very natural way of defining them, but hard to formalize. We do not give the details here, but the idea under this approach will be present in this paper. The equivalence of this definition of shape theory and the usual one is shown in \cite{San}. The importance of this theory lies on the fact that it is internal. That is, we do not make use of external elements (such as the Hilbert cube or polyhedra) to describe the morphisms, as in other shape theories. We just use maps between the metric compacta to define the morphisms. This multivalued theory of shape was reinterpreted later by Alonso-Morón and González Gómez in \cite{MGhomotopical} by the use of hyperspaces with the upper semifinite topology, which seems turns out to be a very adequate topology for our purposes. We define it in the following section.
\subsection{Hyperspaces with the upper semifinite topology}\label{hyperspaces}
The idea of hyperspaces is to define new spaces from old with a related topology. As a general reference for hyperspaces, we recommend the paper \cite{Mtopologies} and the book \cite{Nhyperspaces}.

Given a topological space $X$, we define the \emph{hyperspace} of $X$ as the set of its non-empty closed subsets,
\[2^X=\{C\subset X: C\textrm{ is closed }\}.\]
There are some distinguished elements of $2^X$: The subset $X$ is always a closed subspace of $X$, so it is a point $X\in 2^X$, sometimes named the \emph{fat} point. If $X$ is $T_1$, then every point is closed, so we can consider every singleton $\{x\}$, with $x\in X$, as a point $\{x\}\in 2^X$. The subset 
\[\left\lbrace\{x\}:x\in X\right\rbrace\subset 2^X,\]
is the \emph{canonical copy} of $X$ in $2^X$. It is the image of the inclusion map
\mapeo{\phi}{X}{2^X_H}{x}{\{x\}.} 

We can endow hyperspaces with a number of topologies. If $\cms$ is a compact metric space, the most common topology for the hyperspace is the one induced by the \emph{Hausdorff distance}, defined for two points $C,D\in 2^X$ as
\[d_H(C,D)=\inf\left\lbrace\varepsilon>0: C\subset D_{\varepsilon}, D\subset C_{\varepsilon}\right\rbrace,\]
where 
\[C_{\varepsilon}=\conjunto{x\in X:\dist{x}{C}<\varepsilon}\]
is the \emph{generalized ball} of radius $\varepsilon$. With this metric, $2^X_H=(2^X,\textrm{d}_H)$ is a compact metric space and the inclusion map $\phi$ is an isometry. That means, in particular, that the canonical copy $\phi(X)$ is homeomorphic to the original space $X$ or, in other words, $X$ is embedded in $2^X_H$. More results about hyperspaces with the Hausdorff metric and its relations with the base space can be seen in \cite{MGthe}.

We next define a more general topology for hyperspaces that will be used widely in this paper. It is a very easy-to-use topology but, on the other hand, the hyperspace has very poor topological properties (for example, in general it will be a non-Hausdorff topology). Let $X$ be a topological space. For every open set $U\subset X$, define $$B(U)=\conjunto{C\in 2^X: C\subset U}\subset 2^X.$$ The family $$B=\{B(U):U\subset X\enspace\textrm{open}\}$$ is a base for the \emph{upper semifinite topology} for the hyperspace $2^X$, which will be writen $2^X_u$. The closure operator of this topology is very easy to describe: Given a $T_1$ space $X$ and $C\in 2^X$, the closure of the set constisting of just this point is
$$\overline{\{C\}}=\conjunto{D\in 2^X:C\subset D}.$$
The general references for hyperspaces contain only the definition and some properties for this topology. We add two more references \cite{MGupper,MGhomotopical} about this topology and some of its properties, that will be used here. We list some of them here.
\begin{prop}Let $X,Y$ be Tychonoff spaces. We have the following.
\begin{itemize}[noitemsep]
\item[$i)$] The set $X$ is the unique closed point in $2^X_u$.
\item[$ii)$] The space $2^X_u$ is a compact connected space.
\item[$iii)$] $X$ is homeomorphic to $Y$ if and only if $2^X_u$ is homeomorphic to $2^Y_u$.
\item[$iv)$] If $X$ is non-degenerate\footnote{Actually, $X$ just need to be a non-degenerate $T_1$ space to satisfy this property.}, $2^X_u$ is a $T_0$ but not $T_1$ space.
\end{itemize}
\end{prop}
In this context, we also have that, if $X$ is a $T_1$ space the inclusion map
\mapeo{\phi}{X}{2^X_u}{x}{\{x\},} is a topological embedding.

In the case of metric compacta, we have some extra properties. Let $\cms$ be a compact metric space. Consider, for every $\varepsilon>0$, the subspace of $2^X$ consisting of all the closed subsets of $X$
$$U_{\varepsilon}=\conjunto{C\in 2^X:\diam(C)<\varepsilon}.$$ The following result is key in the use of the upper semifinite topology for hyperspaces in this text and its proof can be found in \cite{MGhomotopical}.
\begin{prop}
The family $U=\{U_{\varepsilon}\}_{\varepsilon>0}$ is a base of open neighborhoods of the canonical copy $\phi(X)$ inside $2^X_u$.
\end{prop}
\begin{obs}
This result is also shown for the hyperspace $2^X_H$ with the Hausdorff metric in \cite{MGthe}.
\end{obs}
\begin{obs}
Note that if we consider any decreasing and tending to zero sequence of positive real numbers $\{\varepsilon_n\}_{\todon}$, we have that $\{U_{\varepsilon_n}\}_{\todon}$ is a nested countable base of $\phi(X)$ in $2^X_u$.
\end{obs}
Consider a continuous map between compact metric spaces $f:X\rightarrow Y$. We define the \emph{elevation induced} by $f$ as the function $2^f:2^X_u\rightarrow 2^Y_u$ defined in the natural way: For $C\in 2^X_u$, $2^f(C)=\bigcup_{c\in C}f(c)$. This is a continuous\footnote{In weaker topological assumptions for the spaces $X$ and $Y$, this is not always true.} map. Moreover, for every map from a topological space to a hyperspace (of the same space or a different one), we can consider an extension to the whole hyperspace. Let $X, Y$ be compact metric spaces. If $f:X\rightarrow 2^Y_u$ is a continuous map, its \emph{extension} is the function $F:2^X_u\rightarrow 2^Y_u$, given by $$F(C)=\bigcup_{x\in C}f(x).$$ It is an extension in the sense that we can consider that $f$ is actually a continuous map from the canonical copy of $X$ in $2^X_u$. That is, strictly speaking, $F$ would be the extension of the map $f^*:\phi(X)\rightarrow 2^Y_u$, with $f^*(\{x\})=f(x)$, which is continuous because $f$ is.
This is Lemma 3 in \cite{MGhomotopical}.
\begin{lem}[Continuity of the extension map]\label{teo:extension}
The extension of every continuous map $f:X\rightarrow 2^Y_u$ is well defined and continuous.
\end{lem}
\subsection{Finite spaces and the Alexandroff-McCord correspondence}
Alexandroff spaces are topological spaces satisfying a topological condition that makes them very special. This notion was introduced by Alexandroff in \cite{Adiskrete}. A topological space $X$ is said to be \emph{Alexandroff} provided arbitrary intersections of open sets are open. Obviously, the most important case of Alexandroff spaces are the \emph{finite} topological ones. Many of the hyperspaces considered in our results are finite. A good reference for Alexandroff and finite topological spaces are the notes of May \cite{Mfinitetopological,Mfinitecomplexes}. We also recommend two papers about Alexandroff and finite spaces \cite{Sfinite,Msingular} that were essential in its development. Finite topological spaces have captured a lot of attention in the last years because of the developments of digital and computational topology. In a series of papers, Barmak and Minian have shown very interesting theorems about the algebraic topology of finite topological spaces (for example, generalizating notions such as collapsibility and simple homotopy type to finite topological spaces). See, for instance, \cite{BMsimple,BMone-point,BMautomorphism} or Barmak's book \cite{Balgebraic}.  One could have the intuition that a topological space with a finite set of points cannot contain a deep geometric information, but this is shown not to be the case. Concerning Alexandroff spaces, is good to have in mind finite topological spaces, for simplicity. We can not require too strong separation properties to Alexandroff spaces, because they would turn trivial: An Alexandroff $T_1$ space is discrete. But, on the other hand, finite $T_0$ spaces have some geometric interest, since they have, at least, one closed point. Moreover, in terms of algebraic topology, we can consider only Alexandroff $T_0$ spaces because of the following theorem of McCord \cite{Msingular}.
\begin{teo}
Let $X$ be an Alexandroff space. There exists a quotient $T_0$ space $q_X:X\rightarrow X_0$ homotopically equivalent to $X$ ($q$ is a homotopy equivalence). Moreover, for every map between Alexandroff spaces, $f:X\rightarrow Y$ there is a unique map $f_0:X_0\rightarrow Y_0$, between $T_0$ Alexandroff spaces, such that $q_Yf=f_0q_X$.
\end{teo}
\paragraph{Alexandroff spaces and posets}
The most important property of an Alexandroff space $X$ is that it has a distinguished basis. For every $x\in X$, we can consider the intersection $$B_x=\bigcap_{x\in U\enspace\textrm{open}}U$$ of all the open sets containing $x$, which is open and it is called the \emph{minimal neighborhood} of $x$, because, by definition, it is contained in every open set containing $x$. It can be shown, that the set of minimal neighborhoods, $\conjunto{B_x:x\in X}$ is a base for the topology of $X$, called the \emph{minimal basis} of $X$. This minimal basis defines a reflexive and transitive relation on the space $X$. For $x,y\in X$, say $x\leqslant y$ if $B_x\subset B_y$. This relation is a partial order if and only if $X$ is $T_0$. On the other hand, every reflexive and transitive relation on a set $X$ determines an Alexandroff topology, with basis the sets $U_x=\{y\in X:y\leqslant x\}$. So, we have the following correspondence.
\begin{prop}
For every set, its Alexandroff topologies are in bijective correspondence with its reflexive and transitive relations. The topology is $T_0$ if and only if the relation is a partial order.
\end{prop}
We call a set $X$ with a partial order $\leqslant$ a \emph{poset}. The last proposition tells us that Alexandroff $T_0$ spaces (sometimes called A-spaces) and posets are the same thing. In what follows we will use both points of view without distinction. With this notation, continuous maps are easily characterized. A function $f:X\rightarrow Y$ of Alexandroff spaces is continuous if and only if is order preserving, that is, $x\leqslant y$ implies $f(x)\leqslant f(y)$.
\paragraph{Alexandroff-McCord correspondence}
We recall the correspondence proved by McCord \cite{Msingular} (we call it the Alexandroff-McCord correspondence because it was Alexandroff who first worked on it) in which simplicial complexes are related with Alexandroff $T_0$ spaces. A \emph{simplicial complex} is a set of \emph{vertices} $V$ and a finite set of \emph{simplexes} $K\subset 2^V$ satisfying that any subset of a simplex $\tau\subset\sigma\in K$ is a simplex $\tau\in K$. A simplicial map $f:K\rightarrow L$ is a function between simplicial complexes $K$ and $L$ sending vertices to vertices (and hence simplexes to simplexes). A \emph{polyhedron} $X$ is a topological space obtained as the realization of a simplicial complex $K$ as a subset of an Euclidean space $X=|K|$ (see references for details). Every simplicial map $f:K\rightarrow L$ defines a continuous map between its realizations $|f|:|K|\rightarrow |L|$ turning combinatorics into topology. Given an A-space space $X$, define $\mathcal{K}(X)$ as the abstract simplicial complex having as vertex set $X$ and as simplices the finite totally ordered subsets $x_0\leqslant\ldots\leqslant x_s$ of the poset $X$. A continuous map $f:X\rightarrow Y$ of A-spaces defines a simplicial map $\mathcal{K}(f):\mathcal{K}(X)\rightarrow\mathcal{K}(Y)$, since it is order preserving. Now, we can define the following map $\psi=\psi_X:|\mathcal{K}(X)|\rightarrow X$ as follows. Every point $z\in|\mathcal{K}(X)|$ is contained in the interior of a unique simplex $\sigma$ spanned by a strictly increasing finite sequence $x_0\leqslant x_1\leqslant\ldots\leqslant x_s$ of points of $X$. We define $\psi(z)=x_0$, and the following theorem holds.
\begin{teo}[McCord \cite{Msingular}]\label{teo:AMcorrespondence}
The map $\psi_X$ is a weak homotopy equivalence. Moreover, given a map $f:X\rightarrow Y$ of A-spaces, the induced simplicial map $\mathcal{K}(f)$ makes the following diagram commutative.
$$\xymatrix@C=1.7cm{
X\ar[r]^{f}			 						  											& Y\\
|\mathcal{K}(X)|\ar[u]^{\psi_X}\ar[r]_{|\mathcal{K}(f)|}		    & |\mathcal{K}(Y)|\ar[u]_{\psi_Y}
}$$
\end{teo}
\begin{ej}
Consider the finite space $X=\{a,b,c,d\}$ with proper open sets 
$$\tau=\conjunto{\{a\},\{c\},\{a,c\},\{a,b,c\},\{a,c,d\}}.$$ Its minimal basis is $$\conjunto{B_a=\{a\},B_b=\{a,b,c\},B_c=\{c\},B_d=\{a,c,d\}}.$$ Hence, $X$ is a poset with $a\leqslant b,d$, $c\leqslant b,d$. The corresponding simplicial complex $\mathcal{K}(X)$ has vertices $a,b,c,d$ and simplices 
$\langle a,b\rangle, \langle a,d\rangle, \langle c,b\rangle, \langle c,d\rangle$, whose realization is homeomorphic to a sphere $\mathbb{S}^1$. Hence $X$ has the homotopy and singular homology groups of $\mathbb{S}^1$.
\end{ej}
On the other direction, given a simplicial complex $K$, we can define an A-space $\mathcal{X}(K)$ whose points are the simplices of $K$ and the relation is given as $\sigma\leqslant\tau$ if and only if $\sigma\subset\tau$ as simplices. Also, from any simplicial map $g:K\rightarrow L$ it is evident that we obtain a continuous map $\mathcal{X}(g):\mathcal{K}\rightarrow\mathcal{L}$ of A-spaces. Now, since $\mathcal{X}(K)$ in an A-space, we can apply the previous theorem to obtain the simplicial complex $\mathcal{K}(\mathcal{X}(K))=K'$ and the weak homotopy equivalence $$\phi_K=\psi_{\mathcal{X}(K)}:|K|=|K'|=|\mathcal{K}(\mathcal{X}(K))|\longrightarrow\mathcal{X}(K).$$ Again, for every simplicial map $g:X\rightarrow Y$ we have that the following diagram commutes up to homotopy.
$$\xymatrix@C=2cm{
|K|\ar[r]^{g}\ar[d]_{\phi_K}			 		& |L|\ar[d]^{\phi_L}\\
\mathcal{X}(K)\ar[r]_{\mathcal{X}(g)}	& \mathcal{X}(L)
}$$
So, there is a mutual correspondence of simplicial complexes and A-spaces (or posets) preserving homotopy and singular homology groups. Note that this means that there are A-spaces with the same homotopy and singular homology groups as every possible simplicial complex. Concretely, there are finite $T_0$ spaces with the same homotopy and singular homology groups as any compact polyhedron.

Note that, given a simplicial complex $K$, we can apply the correspondence of theorem \ref{teo:AMcorrespondence} sequentially to obtain $\mathcal{K}(\mathcal{X}(\overset{n}{\cdots}\mathcal{K}(\mathcal{X}(K))))=K^{(n)}$ the $n$-th barycentric subdivision of $K$. Similarly, given any A-space $X$, we can apply the correspondence $n$ times to obtain what we will call the \emph{$n$-th barycentric subdivision} $X^{(n)}=\mathcal{K}(\mathcal{X}(\overset{n}{\cdots}\mathcal{K}(\mathcal{X}(K))))$ of the A-space $X$.

\subsection{Persistent homology}\label{persistence}
In the last years, the fields of Computational Topology and Applied Algebraic Topology have had a great and successful development. The deep and abstract mathematical concepts and theorems of (Algebraic) Topology have been shown as a very useful tool in real world problems, so the interest of other areas of science in them, is becoming bigger and bigger. As general references for these topics we give the books, \cite{Ztopology,EHcomputational,Gelementary}. We are interested in the more specific field of Topological Data Analysis. This consists of the study and management of (maybe belonging to real world) data sets using topological constructions and techniques. The excellent surveys \cite{Ctopology} by G. Carlsson and \cite{Gbarcodes} by R. Ghrist are strongly recommended for this topic.

In particular, we recall the powefull tool of persistent homology. Persistence is an algebraic topological technique used to detect topological features in contexts where we have not all the information about the space or the information we have is somehow noisy. We recommend, besides the general references quoted, the surveys \cite{EHpersistent,Vsketches}. It is usually agreed that the concept of persistence born in three different ways: Frosini and Ferri's group, studying the persistence of $0$-dimensional homology of functions (using the concept of size function) \cite{Fsize}, Vanessa Robins introducing the concept of persistent Betti numbers in a shape theory context to understand the evolution of homology in fractals \cite{Rtowards} and Edelsbrunner group \cite{ELZtopological}. In this paper, we deal with Robin's approach to persistence with a shape theory perspective, with the use of inverse sequences.
\paragraph{The idea of persistence}
We illustrate the notion of persistent homology through a very schematic example. Consider we have a finite set of points $\mathbb{X}$ (and we know the distances between them), possibly as a noisy sample of an unknown topological space $X$. If we want to detect some topological properties of $X$ from $\mathbb{X}$, one way could be to construct a simplicial complex based on this set of points and study its topologicalproperties. For example, in figure \ref{pointcloud},
\begin{figure}[h!]
\begin{center}
\includegraphics[scale=1]{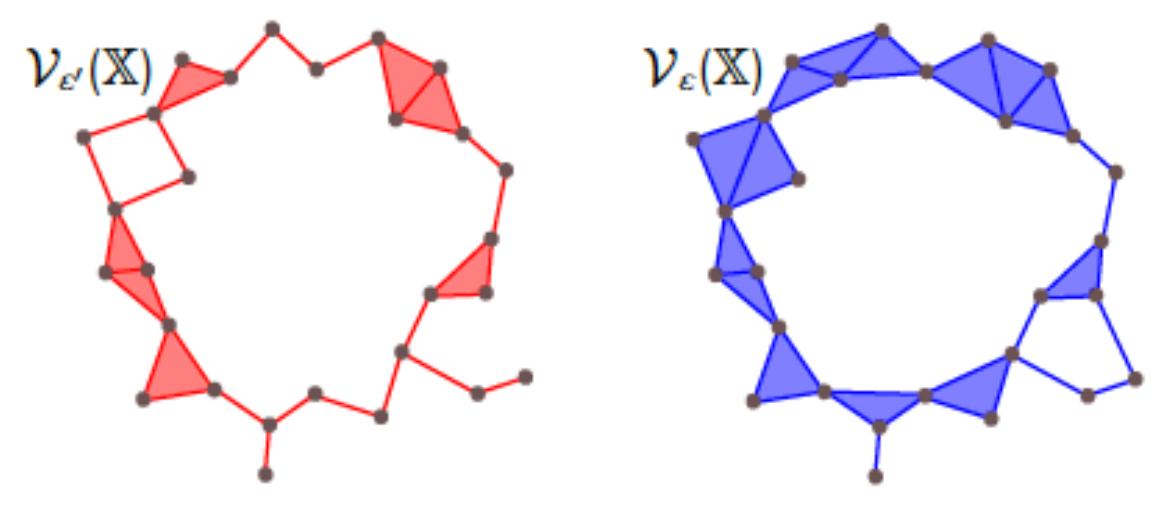}
\end{center}
  \caption{The Vietoris-Rips complexes of a point cloud with two parameters.}
  \label{pointcloud}
\end{figure}
we have the Vietoris-Rips complexes (the \emph{Vietoris-Rips complex} $\mathcal{V}_{\varepsilon}(X)$ of a finite set of points $X$ with parameter $\varepsilon$ in a metric $\textrm{d}$, is the simplicial complex with vertex set $X$ and $\sigma\subset X$ is a simplex of $\mathcal{V}_{\varepsilon}(X)$ iff $\diam(\sigma)<\varepsilon$) of a finite set of points $\mathbb{X}$, which is a noisy sample of an underlying space $X=\mathbb{S}^1$, with two different real parameters $0<\varepsilon'<\varepsilon$. Both detect the main feature of $X$, the central hole or $1$-cycle. But we have that none of them really determine the first homology group of the actual space $X$, because $$H_1(\mathcal{V}_{\varepsilon'}(\mathbb{X});\mathbb{Z})\cong H_1(\mathcal{V}_{\varepsilon}(\mathbb{X});\mathbb{Z})\cong\mathbb{Z}\oplus\mathbb{Z}\ncong\mathbb{Z}\cong H_1(X;\mathbb{Z}).$$ The persistent homology idea is just to consider the inclusion $\mathcal{V}_{\varepsilon'}(\mathbb{X})\hookrightarrow\mathcal{V}_{\varepsilon}(\mathbb{X})$ and the image of the induced map on the first homology groups, that is, 
$$\textrm{Im}\big(H_1(\mathcal{V}_{\varepsilon'}(\mathbb{X});\mathbb{Z})\hookrightarrow H_1(\mathcal{V}_{\varepsilon}(\mathbb{X});\mathbb{Z})\big) \cong\mathbb{Z}\cong H_1(X;\mathbb{Z})$$ which really captures only the desired feature, ignoring the noise of $\mathbb{X}$.
\paragraph{Filtrations}
In general, suppose we have a \emph{filtration}, i.e., a finite sequence of nested simplicial complexes
$$\emptyset=K_0\hookrightarrow K_1\hookrightarrow\ldots\hookrightarrow K_s.$$
We are interested in the topological evolution of the sequence of the homology groups, so, for every $p\in\mathbb{N}$ and every abelian group $G$, we can consider the induced $p$-th homology finite sequence
$$\{0\}=H_p(K_0;G)\hookrightarrow H_p(K_1;G)\hookrightarrow\ldots\hookrightarrow H_p(K_s;G).$$
As we move forward in the sequence, new homology classes can appear and some could merge or vanish. We collect the homology classes as follows. The \emph{$p$-th persistent homology groups} are the images of the homomorphisms induced by inclusion
$$H_p^{ij}=\textrm{Im}\big(H_p(K_i;G)\hookrightarrow H_p(K_j;G)\big)$$ for $0\leqslant i<j\leqslant s$. Similarly, the \emph{$p$-th persistent Betti numbers} are the ranks of these groups $\beta_p^{ij}=\textrm{rank} H_p^{ij}$. We can do the same definitions with reduced homology. The collection of persistent Betti numbers can be visualized in a \emph{persistence diagram}. Given a filtration of simplicial complexes, there are several algorithms determining these numbers and the evolution of the homology classes. See the quoted references for more details. 

There are several ways of arriving to a filtration of simplicial complexes. We mention the main two of them.
\begin{itemize}
\item A finite set of points and its distances. Given any finite metric space $\mathbb{X}$ (as in the previous example), called a \emph{point cloud}, we can produce filtrations of simplicial complexes taking the Vietoris-Rips, \v{C}ech or other complexes of $\mathbb{X}$ for every $\varepsilon>0$. There will be only a finite number of different complexes since $\mathbb{X}$ is finite, so we obtain a filtration of simplicial complexes along the parameter $\varepsilon$.
\item Consider a simplicial complex $K$ and a real valued function $f:K\rightarrow\mathbb{R}$ which is monotonic, meaning that if $\tau$ is a face of $\sigma$, then $f(\tau)\leqslant f(\sigma)$. Then, supposing the different values of the function are $-\infty=a_0<a_1<\ldots<a_s$, if we set $K_i=f^{-1}(-\infty,a_i]$ for $i=1,2,\ldots,s$, we have that $K_i$ are subcomplexes of $K$, $K_i$ is a subcomplex of $K_{i+1}$, for every $i=1,2,\ldots,s-1$, and $K_{a_s}=K$. Thus we have a filtration called the \emph{filtration of the function} $f$.
\end{itemize}
\paragraph{Structure of persistence}
One step further in the study of persistence is to find some structure in the evolution of the homology classes in a given filtration. In this direction, we recall the Structure Theorem by Carlsson and Zomorodian \cite{CZcomputing}. Let $F$ be a field. We define a \emph{persistence module} $\mathcal{M}$ as a family of vector spaces \footnote{The definition still holds if we replace $F$ by a commutative ring with unity, obtaining then $R$ modules $M_i$, but we need this stronger condition for the structure theorem.} $M_i$ over $F$ and homomorphisms $\varphi_i:M_i\rightarrow M_{i+1}$, for $i\in\mathbb{N}$. For example, the induced homology finite sequence of a filtration, where the maps $\varphi$ send a homology class to the one containing it. We will say that $\mathcal{M}$ is of finite type if $M_i$ is a finitely generated $R$-module and there exists an integer $m$ such that $\varphi_i$ is an isomorphism for $i\geqslant m$. Now we define the elements for the classification which, in some sense, represents the beginning and end of an homology class. A \emph{persistence interval} is an ordered pair $(i,j)$, with $0\leqslant i<j$, $i,j\in\mathbb{Z}\cup\{+\infty\}$. A finite set of persistence intervals is called a \emph{barcode}. The following correspondence is stablished.
\begin{teo}[Structure]
The isomorphism classes of persistence modules of finite type over a field are
in bijective correspondence with barcodes.
\end{teo}
The proof of this theorem uses some advanced algebra, including the structure theorem of finitely generated modules and graded modules over \textsc{pid}s, which we do not include here for simplicity. For the algebraic machinery used in the proof, see \cite{DFabstract}. The importance of this result, which gives a structure to the persistence modules, is that it allows us to use the barcodes, a very intuitive way of representing the evolution of the homology classes, because they really determine the persistence module, up to isomorphism. So they are a good way to represent persistence. On the other hand, this result enables us to modify the standard reduction algorithm for homology using the properties of the persistence module to derive a rather simple algorithm to compute the barcodes. This is implemented in the Matlab routine Plex and in some functions of the library \textsc{phat} in \textbf{R}.
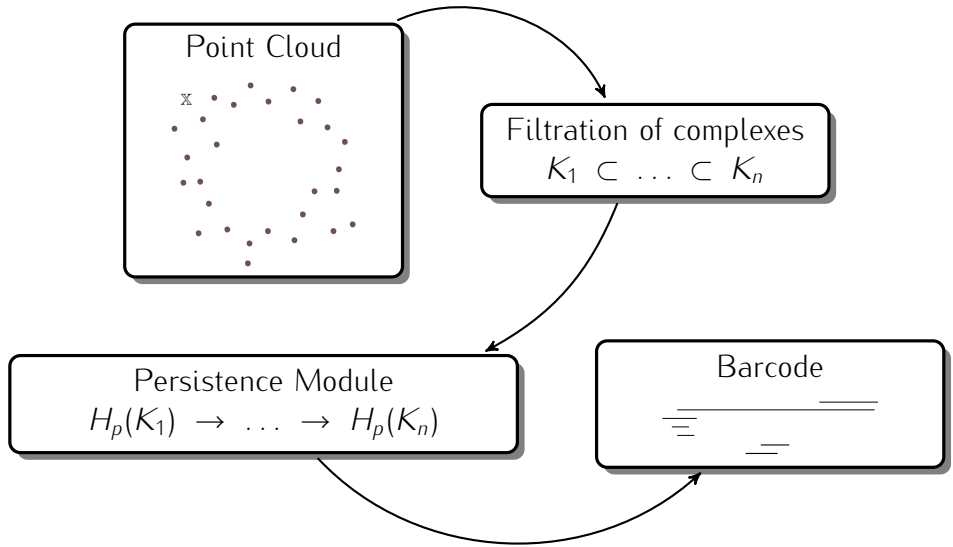
\begin{figure}
\begin{center}
\begin{tikzpicture}[node distance=1cm, auto,]
 \node[punkt] (pointcloud) {Point Cloud\\ \vspace{0.3cm}\includegraphics[scale=0.25]{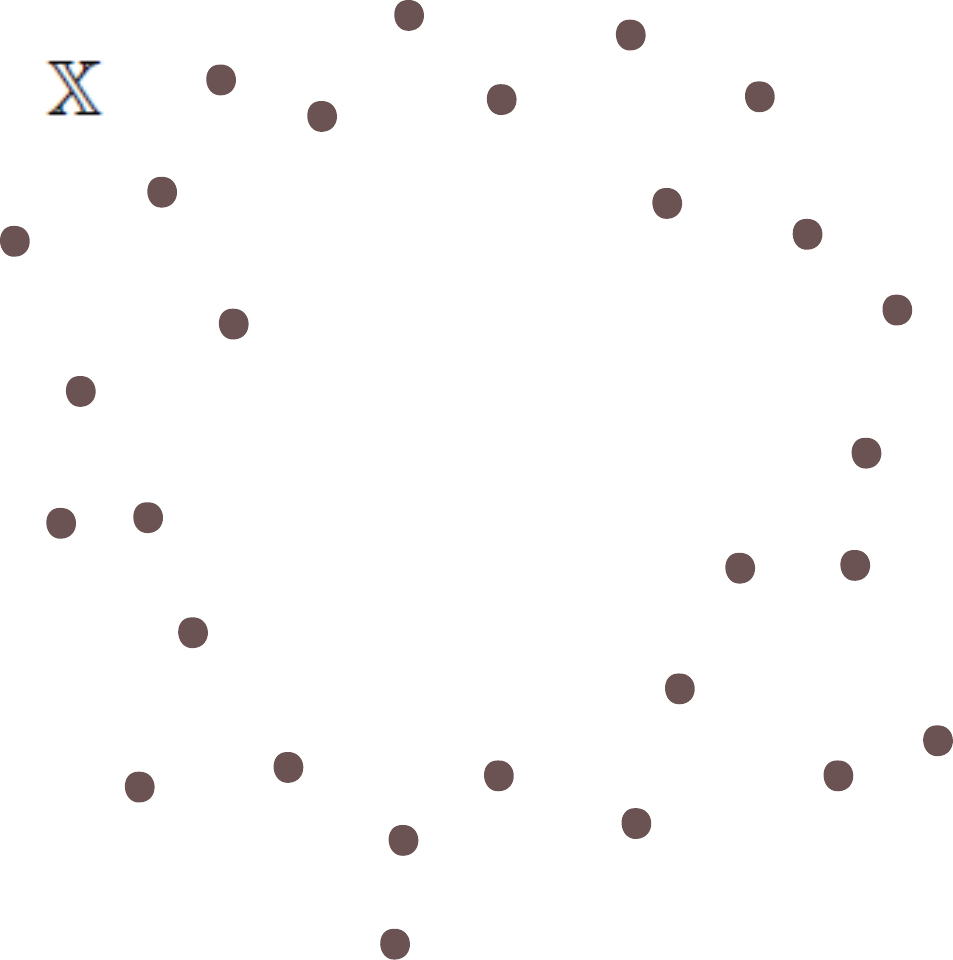}};
 \node[punkta, inner sep=5pt,right=1cm of pointcloud](filtration) {Filtration of complexes\\$K_1\subset\ldots\subset K_n$};
 \node[punktaa, inner sep=5pt,below=1cm of pointcloud](module) {Persistence Module\\$H_p(K_1)\rightarrow\ldots\rightarrow H_p(K_n)$};
 \node[punktaaa, inner sep=5pt,right=1cm of module](barcode) {Barcode\\ \vspace{0.3cm}\includegraphics[scale=0.40]{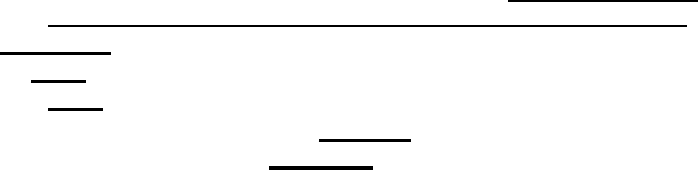}};
  \draw [pil,bend left=45](pointcloud) edge (filtration);
  \draw [pil,bend left=20](filtration) edge (module);
  \draw [pil,bend right=45](module) edge (barcode);
\end{tikzpicture}
\end{center}
\caption{The process of obtaining a barcode from a point cloud}
\end{figure}

\section{The Main Construction}\label{mainconstruction}
Let us recall the main construction introduced in section 6 of \cite{MCLepsilon}. There, given a compact metric space, it is obtained an inverse sequence of finite approximations of the space and some sequences of real numbers that allow us to define continuous maps between the approximations. 

Let $(X,\textrm{d})$ be a compact metric space and $\varepsilon>0$ a positive real number. A finite subset $A\subset X$ is said to be a \emph{finite $\varepsilon$-approximation} of $X$ if, for every $x\in X$, there is at least one point $a\in A$ such that $\dist{x}{a}<\varepsilon$. 

\begin{obs}Compact metric spaces have finite $\varepsilon$-approximations for every $\varepsilon>0$.
\end{obs}

Given a non-empty finite subset $A\subset X$ of a compact metric space $(X,\textrm{d})$, we consider, for each point $x\in X$, the \emph{set of closest points of $A$} as the subset of the hyperspace of $A$ consisting of the points minimizing the distance to $x$:
$$A(x)=\conjunto{a\in A:\dist{x}{a}=\dist{x}{A}}\subset 2^A_u\subset 2^X_u.$$
Note that $A$ is a discrete finite subset of $X$, hence the topology of $2^A_u$ as a subespace of $2^X_u$ is a finite space with the relation $\subset$.
It is natural, then, to define a function from the space to its closest sets. We will call the \emph{nearby map} from $X$ to $A$ to the function $$q_A:X\rightarrow 2^A_u,$$ defined by $q_A(x)=A(x)$. The extension of the nearby map will be usually written as $$r_A:2^X_u\rightarrow 2^X_u.$$
Both will be shown to be continuous maps because of the following lemma.
\begin{lem}\label{teo:alpha}
Let \cms be a compact metric space and $A\subset X$ a finite subset. For every $x\in X$ there exists $\delta>0$ such that, for every $y\in\textrm{B}(x,\delta)$, $A(y)\subset A(x)$.
\end{lem}
\begin{proof}
Let $x\in X$ and consider the distances $\delta^-=\dist{x}{A}\geqslant 0$ and $\delta^+=\dist{x}{A\setminus A(x)}>0$ (if $A\setminus A(x)=\emptyset$, then $A(x)=A$ and the result is obvious, so we will assume that it is not empty). Now, fix $$\delta=\frac{\delta^+-\delta^-}{2}>0.$$  If $a\in A(x)$ and $b\in A\setminus A(x)$, we see that, for every $y\in\textrm{B}(x,\delta)$,
\begin{align*}
\dist{y}{a}&\leqslant\dist{y}{x}+\dist{x}{a}<\delta+\delta^-=\frac{\delta^++\delta^-}{2},\\
\delta^+\leqslant\dist{x}{b}&\leqslant\dist{x}{y}+\dist{y}{b}<\delta+\dist{y}{b}.
\end{align*}
Whence $$\dist{y}{b}>\frac{\delta^++\delta^-}{2}>\dist{y}{a},$$ so $A(y)\subset A(x)\qedhere$
\end{proof}
As an immediate corollary, we obtain the continuity of the nearby map and its extension.
\begin{cor}\label{teo:nearbydistance}
Let \cms be a compact metric space and $A\subset X$ a finite subset. The nearby map $q_A:X\rightarrow 2^X_u$ is continuous. Hence its extension $r_A$ is also continuous.
\end{cor}
\begin{proof}
The map $q_A$ satisfies that, for every $x\in X$, there exists $\delta>0$ such that $$q_A(\textrm{B}(x,\delta))\subset q_A(x),$$ hence $q_A$ is continuous$\qedhere$
\end{proof}
\begin{obs}
If $A$ is a finite $\varepsilon$-approximation of a compact metric space $(X,\textrm{d})$, then the images of the points $x\in X$ by the nearby map are sent to the subspace $$U_{2\varepsilon}(A)=\{C\in 2^A:\diam{C}<2\varepsilon\}\subset 2^A_u,$$ because of the triangle inequality. That is, the nearby map may be written as $q_A:X\rightarrow U_{2\varepsilon}(A)$.
\end{obs}

Let \cms be a compact metric space. Given a real number $\varepsilon>0$, and a finite $\varepsilon$-approximation $A$, we say that $0<\varepsilon'<\varepsilon$ is \emph{adjusted} to $A$ whenever 
$$\varepsilon'<\frac{\varepsilon-\gamma}{2},$$ where $$\gamma=\sup\conjunto{\dist{x}{A} : x\in X}$$ being the supremum of the distances of points of $X$ to $A$ which obviously satisfies $\gamma<\varepsilon$.

The following result is the more important concerning the Main Construction. It says that, given any finite approximation of a compact metric space, we always can find a tighter finite approximation of the space and define nearby continuous maps between some finite spaces based on these approximations.
\begin{lem}\label{teo:adjustedapp}
Let \cms be a compact metric space and consider a real number $\varepsilon>0$ and a finite $\varepsilon$-approximation $A$ of $X$. For every $0<\varepsilon'<\varepsilon$ adjusted to $A$ and every finite $\varepsilon'$-approximation $A'$, the map $p:U_{2\varepsilon'}(A')\rightarrow U_{2\varepsilon}(A)$, defined by $p(C)=r_{A}(C)$, is well defined and continuous.
\end{lem}

By induction, we can repeat the process indefinitely, obtaining:
\begin{teo}[Main Construction]\label{teo:mainconstruction}
For every compact metric space $(X,\textrm{d})$, there exist a decreasing sequence of positive real numbers $\{\varepsilon_n\}_{\todon}$ tending to zero, a sequence $\{A_n\}_{\todon}$ of finite $\varepsilon_n$-approximations of $X$ with $\varepsilon_{n+1}$ adjusted to $A_n$ for every $\todon$ and continuous maps $p_{n,n+1}:U_{2\varepsilon_{n+1}}(A_{n+1})\rightarrow U_{2\varepsilon_n}(A_n)$ for every $\todon$.
\end{teo}

Hence, we get an inverse sequence of finite spaces and continuous maps:
$$U_{2\varepsilon_1}(A_1)\xleftarrow{\enspace p_{1,2}\enspace} U_{2\varepsilon_2}(A_2)\xleftarrow{\enspace p_{2,3}\enspace}\ldots\xleftarrow{\enspace p_{n-1,n}\enspace} U_{2\varepsilon_n}(A_n)\xleftarrow{\enspace p_{n,n+1}\enspace}U_{2\varepsilon_{n+1}}(A_{n+1})\xleftarrow{\enspace p_{n+1,n+2}\enspace}\ldots$$

An inverse sequence $\{U_{2\varepsilon_n}(A_n),p_{n,n+1}\}$ with $\{\varepsilon_n\}_{\todon}$ converging to zero, satisfying that for every $\todon$, $A_n$ is a finite $\varepsilon_n$-approximation of $X$ with $\varepsilon_{n+1}$ adjusted to $A_n$ will be called a \emph{finite approximative sequence} (\textsc{fas}) of $X$.

\begin{obs}
Observe that, given a compact metric space, this process is completely constructive. We can compute all the real numbers and select finite approximations that satisfy the quoted properties. Being an inductive process, we compute the numbers and approximations in their strictly necessary order.
\end{obs}

\begin{obs}
Strictly speaking, a \textsc{fas} will be the inverse sequence of finite spaces quoted above. But we will use \textsc{fas} to make reference also to the approximations and the numbers obtained, $\fas$, because they determine uniquely the finite spaces and maps of the inverse sequence.
\end{obs}

\begin{obs}
Theorem \ref{teo:mainconstruction} states that every compact metric space has a \textsc{fas}. In general, \textsc{fas} are not unique.
\end{obs}

We can now use the Alexandroff-McCord correspondence to obtain an inverse sequence of polyhedra. For every $n\in\mathbb{N}$ and finite $\textrm{T}_0$ space $U_{2\varepsilon_n}(A_n)$, there exists a simplicial complex $\mathcal{K}(U_{2\varepsilon_n}(A_n))$ with vertex set the points $D\in U_{2\varepsilon_n}(A_n)$ and simplexes $\langle D_0,D_1,\ldots,D_s\rangle$ with $D_0\subset D_1\subset\ldots\subset D_s$ such that there is a weak homotopy equivalence between the finite space and the geometric realization of the simplicial complex $$f_n:|\mathcal{K}(U_{2\varepsilon_n}(A_n))|\longrightarrow U_{2\varepsilon_n}(A_n),$$ defined as follows. Every point $x\in|\mathcal{K}(U_{2\epsilon_n}(A_n))|$ is contained in the interior of a unique simplex $\sigma=\langle D_0,D_1,\ldots,D_s\rangle$, so we can define $f_n(x)=D_0$.

We also have simplicial maps\footnote{Following McCords paper's notation we should write $\mathcal{K}(p_{n,n+1})$ for the simplicial maps but we will omit this notation, using the same as for the maps between the finite spaces, $p_{n,n+1}$, for the sake of simplicity.} between the polyhedra, defined on the vertices and extended as usual to simplices: 
\begin{eqnarray*}
p_{n,n+1}:\mathcal{K}(U_{2\epsilon_{n+1}}(A_{n+1})) & \longrightarrow & \mathcal{K}(U_{2\epsilon_n}(A_n)) \\
D & \longmapsto & p_{n,n+1}(D) \\
\langle D_0, D_1, \ldots, D_s\rangle & \longmapsto & \langle p_{n,n+1}(D_0), p_{n,n+1}(D_1), \ldots, p_{n,n+1}(D_s)\rangle
\end{eqnarray*} 
where, if $$D_0\subset D_1\subset\ldots\subset D_s,$$ then $$p_{n,n+1}(D_0)\subset p_{n,n+1}(D_1)\subset\ldots\subset p_{n,n+1}(D_s).$$ The realizations of these simplicial maps satisfy that, for every $n\in\mathbb{N}$, the diagram 
$$\xymatrix@C=2cm{|\mathcal{K}(U_{2\varepsilon_n}(A_n))|\ar[d]_{f_n} & |\mathcal{K}(U_{2\varepsilon_{n+1}}(A_{n+1}))|\ar[d]^{f_{n+1}}\ar[l]_{|p_{n,n+1}|}\\
U_{2\varepsilon_n}(A_n) & U_{2\varepsilon_{n+1}}(A_{n+1})\ar[l]^{p_{n,n+1}}}$$ 
commutes. So we obtain an inverse sequence of polyhedra and a map between the inverse sequences of finite spaces and polyhedra.
\begingroup\makeatletter\def\f@size{10}\check@mathfonts
$$\xymatrix{|\mathcal{K}(U_{2\varepsilon_1}(A_1))|\ar[d]_{f_1} & |\mathcal{K}(U_{2\varepsilon_{2}}(A_{2}))|\ar[l]_{|p_{1,2}|}\ar[d]_{f_2}&\ldots\ar[l]&|\mathcal{K}(U_{2\varepsilon_n}(A_n))|\ar[d]_{f_n}\ar[l] & |\mathcal{K}(U_{2\varepsilon_{n+1}}(A_{n+1}))|\ar[d]^{f_{n+1}}\ar[l]_{|p_{n,n+1}|}&\ldots\ar[l]\\
U_{2\varepsilon_1}(A_1) & U_{2\varepsilon_{2}}(A_{2})\ar[l]^{p_{1,2}}&\ldots\ar[l]&U_{2\varepsilon_n}(A_n)\ar[l] & U_{2\varepsilon_{n+1}}(A_{n+1})\ar[l]^{p_{n,n+1}}&\ldots\ar[l]}$$ 
\endgroup
Every inverse sequence of polyhedra $\left\lbrace|\mathcal{K}(U_{2\varepsilon_{n}}(A_{n})),|p_{n,n+1}|\right\rbrace$ obtained in this way is called a \emph{Polyhedral Approximative Sequence} (or \textsc{pas}) of the space $X$.

The analysis of these inverse sequences is proposed in \cite{MCLepsilon} as a reconstruction process of topological properties of a compact metric space. We will show here that every \textsc{fas} is able to recover the homotopy type of the original space. The shape type reconstruction using \textsc{pas} is show in \cite{preprintShape}. We propose the use of this sequences as an alternative way of producing persistence modules from a point cloud in a process that will be called inverse persistence.
\section{Homotopical reconstruction by \textsc{fas}}\label{homotopicalreconstruction}
In this section, we prove the main result concerning the reconstruction of compact metric spaces. It asserts that in the inverse limit of every \textsc{fas} we can find the homotopy structure of our space.
\begin{teo}\label{teo:limitefinito}
Let $X$ be a compact metric space. The inverse limit of every \textsc{fas} $$\mathcal{X}=\varprojlim\{\U{n},p_{n,n+1}\}$$ contains a subspace $\mathcal{X}^*\subset\mathcal{X}$ homeomorphic to $X$ which is a strong deformation retract of $\mathcal{X}$.
\end{teo}

In order to prove this result, we need some technical lemmas about the above construction. Consider a \textsc{fas} $\{\U{n},p_{n,n+1}\}$ of a compact metric space $X$ and the sequences of numbers with usual notation $\fas$. For every $n\in\mathbb{N}$, we write $\overline{\varepsilon_n}=\frac{\varepsilon_n+\gamma_n}{2}$ and $\underline{\varepsilon_n}=\frac{\varepsilon_n-\gamma_n}{2}$. They clearly satisfy $\overline{\varepsilon_n},\enspace\underline{\varepsilon_n}<\varepsilon_n$ and $\overline{\varepsilon_n}+\underline{\varepsilon_n}=\varepsilon_n$.
\begin{lem}
For every $n<m$, we have $$\sum^m_{l=n}\gamma_{l}<\overline{\varepsilon_n}.$$
\end{lem}
\begin{proof}
For every $n\in\mathbb{N}$, we have that $\varepsilon_{n+1}<\frac{\varepsilon_n-\gamma_n}{2}$, so $\gamma_n+\varepsilon_{n+1}<\gamma_n+2\varepsilon_{n+1}<\varepsilon_n$. \\Now, let $n<m$ be natural numbers, if we write $m=n+k,\enspace k>0$, we can apply the previous observation inductively to obtain
\begin{eqnarray}
\sum^m_{l=n}\gamma_{l}=\sum_{i=0}^k\gamma_{n+i}&<&\left(\sum^{k-1}_{i=0}\gamma_{n+i}\right)+\varepsilon_{n+k}<\left(\sum_{i=0}^{k-2}\gamma_{n+i}\right)+\varepsilon_{n+k-1}<\ldots<\nonumber\\
&<&\gamma_n+\varepsilon_{n+1}<\gamma_n+\frac{\varepsilon_n-\gamma_n}{2}=\frac{\varepsilon_n+\gamma_n}{2}\qedhere\nonumber\qedhere
\end{eqnarray}
\end{proof}

\begin{obs}
The previous lemma gives us a bound in terms of the lower term, so it is readily seen that the infinite sum converges, and $$\sum_{l=n}^{\infty}\gamma_l<\varepsilon_n.$$
\end{obs}

\begin{lem}
For every $n>1$, $\varepsilon_{n}<\frac{\varepsilon_1}{2^{n-1}}$.
\end{lem}
\begin{proof}
We proceed by induction over $n$. The first case is clear, $\varepsilon_2<\frac{\varepsilon_1-\gamma_1}{2}<\frac{\varepsilon_1}{2}$. Now, let us suppose that $\varepsilon_n<\frac{\varepsilon_1}{2^{n-1}}$. Then $\varepsilon_{n+1}<\frac{\varepsilon_n-\gamma_n}{2}<\frac{\varepsilon_n}{2}<\frac{\varepsilon_1}{2^n}\qedhere$
\end{proof}

\begin{prop}
Let $n<m$ be a pair of natural numbers. Let $a_n\in A_n$ and $a_m\in A_m$ be two points of $X$ such that $a_n\in p_{n,m}(\{a_m\})$. Then $\textrm{d}(a_n,a_m)<\overline{\varepsilon_n}$.
\end{prop}
\begin{proof}
Let us write $m=n+k,\enspace k>0$. The relation between the points means that there exists a chain of points between them. That is, there exist $a_{n+1}\in A_{n+1},\ldots,a_{n+k-1}\in A_{n+k-1}$ such that
\begin{eqnarray}
a_n\in p_{n,n+1}(\{a_{n+1}\}),\nonumber\\
a_{n+1}\in p_{n+1,n+2}(\{a_{n+2}\}),\nonumber\\
\ldots\nonumber\\
a_{n+k-1}\in p_{n+k-1,n+k}(\{a_{n+k}\}).\nonumber
\end{eqnarray}
Using the previous proposition, we can now estimate $$\textrm{d}(a_n,a_m)\leqslant\sum_{i=0}^k\textrm{d}(a_{n+l},a_{n+i+1})\leqslant\sum_{i=0}^k\gamma_{n+i}<\overline{\varepsilon_n}\qedhere$$
\end{proof}

Before starting the proof of the theorem, let us reinterpret this inverse limit as sequences of points in $2^X$, with the Haussdorff distance, that is, as sequences of points in $2^X_H$. We consider the inverse limit of the finite spaces. We will write the points of this limit as sequences $\{C_n\}_{n\in\mathbb{N}}\in\mathcal{X}$ ($\{C_n\}$ for short), where, for every $n\in\mathbb{N}$, $C_n\in U_{2\varepsilon_n}(A_n)$, and, for every pair $n<m$, $p_{n,m}(C_m)=C_n$. We have to think about this sequences as sets of points of each $\varepsilon$-approximation, related by a notion of proximity. It turns out that these sequences converge to points of $X$. To have a notion of measure and see this, we will use the Hausdorff distance of the hyperspace $2^X$ of non-empty closed subsets of $X$. It can by characterized(see section \ref{hyperspaces}) in the following way: For $C,D\in 2^X_H$ closed sets of $X$, we will say that the Hausdorff distance of $C$ and $D$ is $\textrm{d}_{\textrm{H}}(C,D)<\varepsilon$ if $C\subset \textrm{B}(D,\varepsilon)$ and $D\subset \textrm{B}(C,\varepsilon)$\footnote{Here, $\textrm{B}(C,\varepsilon)$ is the generalized ball of radius $\varepsilon$, i.e., the set of points $x\in X$ for which there exists a point $c$ of $C$ at distance $\textrm{d}(x,c)<\varepsilon$ or, equivalently, is the union of balls of radius $\varepsilon$ and center any point of $C$, that is, $\textrm{B}(C,\varepsilon)=\bigcup_{c\in C}\textrm{B}(c,\varepsilon)$.}. We are going to prove the following
\begin{prop} Every point of the inverse limit $\{C_n\}\in\mathcal{X}$ is a Cauchy sequence in $2^X_H$ that converges to a singleton $\{x\}$, with $x\in X$. 
\end{prop}
\begin{proof}
First of all, we see that, in terms of the Hausdorff metric, the diference between two elements of the sequence can be bounded in terms of the lower index. Let $\{C_n\}\in\mathcal{X}$ be a point of the inverse limit. Then, the Hausdorff distance between terms of the sequence $C_n$ and $C_m$, with $n<m$, is $\textrm{d}_{\textrm{H}}(C_n,C_m)<\overline{\varepsilon_n}.$ To prove the first condition of the Hausdorff distance, consider $c_n\in C_n$ and $c_m\in C_m$ such that $c_n\in p_{n,m}(\{c_m\})$, and then $\textrm{d}(c_n,c_m)<\overline{\varepsilon_n}$ by the previous lemma. Analogously, for $c_m\in C_m$ we can take $c_n\in p_{n,m}(C_m)$ and the distance satisfies the second condition. 

Now, the sequence of closed sets $\{C_n\}\in\mathcal{X}$ is a Cauchy sequence in $2_H^X$. For any $\varepsilon>0$, it suffices to consider $n_0\in\mathbb{N}$ such that $\varepsilon_{n_0}<\varepsilon$ and then, for every $n,m>n_0$, we have $\textrm{d}_{\textrm{H}}(C_n,C_m)<\varepsilon_{n_0}<\varepsilon.$

It remains to prove that every sequence $\{C_n\}\in\mathcal{X}$ converges to a singleton $\{x\}$ of $X$ in the Hausdorff metric. The sequence is Cauchy in the compact metric (and hence complete) space $2_H^X$, so there exists a unique limit $C\in 2^X$. The diameter of this point of the hyperspace is $$\textrm{diam}(C)=\textrm{diam}(\lim_n C_n)=\lim_n\textrm{diam}(C_n)\leqslant\lim_n 2\varepsilon_n=0$$ because of the continuity of the diameter function regarding to the Hausdorff metric (see \cite{Nhyperspaces}). So $C=\{x\}$, with $x\in X\qedhere$
\end{proof}
\begin{obs}
The meaning of $\{C_n\}\in\mathcal{X}$ converging to a set with only one point $\{x\}\subset X$ is that for every $\varepsilon>0$ there exists $n_0\in\mathbb{N}$, such that, for every $n>n_0$, $\textrm{d}_{\textrm{H}}(\{x\},C_n)<\varepsilon$, i.e., $C_n\subset\textrm{B}(\{x\},\varepsilon)$ and $x\in\textrm{B}(C_n,\varepsilon)$. But, the first condition, meaning $x\in\bigcap_{c\in C_n}\textrm{B}(c,\varepsilon)$, implies the second one, $x\in\bigcup_{c\in C_n}\textrm{B}(c,\varepsilon)$. Henceforth, we will say that $\{C_n\}$ converges to $x$ (written $\{C_n\}\xrightarrow[H]{} x$ or $x=\lim_H \{C_n\}$) for the convergence of $\{C_n\}$ to $\{x\}$ with the Hausdorff metric and we will write $\textrm{d}_{\textrm{H}}(x,C_n)$ for $\textrm{d}_{\textrm{H}}(\{x\},C_n)$, for simplicity.
\end{obs}
We have the following trivial facts relating the Hausdorff distance on the hyperspace of a metric space and the original distance on the space, for distances between points and closed sets.
\begin{prop}
Let $X$ be a metric space, for every pair of points $x,y\in X$ and pair of closed subsets $D\subset C\subset X$, we have:
\begin{itemize}
\item [i)]$\textrm{d}_{\textrm{H}}(x,y)=\textrm{d}(x,y)$.
\item [ii)]$\textrm{d}_{\textrm{H}}(x,C)=\textrm{sup}\left\lbrace\textrm{d}(x,c) : c\in C\right\rbrace
\geqslant\textrm{inf}\left\lbrace\textrm{d}(x,c) : c\in C\right\rbrace=d(x,C)$.
\item [iii)]$\textrm{d}_{\textrm{H}}(x,D)\leqslant\textrm{d}_{\textrm{H}}(x,C)$ but $\textrm{d}(x,D)\geqslant\textrm{d}(x,C)$.
\end{itemize}
\end{prop}
The last property can be interpreted in some sense as a better behaviour of the Hausdorff distance with respect to the upper semifinite topology.
\begin{obs}\label{distancialimite}
We can even bound the distances to the limit. If $\{C_n\}\in\mathcal{X}$ is a point of the inverse limit converging to a point $x\in X$ in the Hausdorff metric, then, for every $n\in\mathbb{N}$, $\textrm{d}_{\textrm{H}}(x,C_n)<\varepsilon_n$. This is so because, if we consider an $m>n$ such that $\textrm{d}_{\textrm{H}}(x,C_m)<\underline{\varepsilon_n}$, then we can write $$\textrm{d}_{\textrm{H}}(x,C_n)\leqslant\textrm{d}_{\textrm{H}}(x,C_m)+\textrm{d}_{\textrm{H}}(C_m,C_n)<\underline{\varepsilon_n}+\overline{\varepsilon_n}=\varepsilon_n.$$
This measure allows us to understand this sequences from another point of view: If $\{C_n\}\in\mathcal{X}$ is such a sequence, then we know there exists an $x\in X$ such that $\{C_n\}$ converges to $\{x\}$ in the Hausdorff metric. But, from the previous remark, we see that, for every $n\in\mathbb{N}$, $x\in\bigcap_{c\in C_n}\textrm{B}(c,\varepsilon_n)$. So, we can see $x$ as the infinite intersection over all natural numbers: $$x=\bigcap_{n\in\mathbb{N}}\left(\bigcap_{c\in C_n}\textrm{B}(c,\varepsilon_n)\right).$$
\end{obs}

\begin{proof}[Proof of Theorem \ref{teo:limitefinito}]
Now, we can define a map $\varphi:\mathcal{X}\rightarrow X$ from the inverse limit $\mathcal{X}$ to the original space $X$. We do this assigning to every sequence $\{C_n\}\in\mathcal{X}$ the unique point $x$ in the limit $x=\lim_H\{C_n\}$. The map $\varphi:\mathcal{X}\rightarrow X$, sending $\{C_n\}$ to $x$ is continuous. Let $\{C_n\}\in\mathcal{X}$ such that $x=\lim_H\{C_n\}$. Then, consider a neighborhood $U$ of $x$ inside $X$. Now we want to find a neighborhood of $\{C_n\}$ in $\mathcal{X}$ with image contained in $U$. There exists an $\varepsilon>0$ such that $x\in\textrm{B}(x,\varepsilon)\subset U$.  Let us consider $n_0\in\mathbb{N}$ such that, for every $n\geqslant n_0$, $\varepsilon_n<\frac{\varepsilon}{2}$. We claim that the basic open neighborhood of the inverse limit $\mathcal{X}$, $$V=\left(2^{C_1}\times 2^{C_2}\times\ldots\times2^{C_{n_0}}\times U_{2\varepsilon_{n_0+1}}(A_{n_0+1})\times\ldots\right)\cap\mathcal{X}$$ is the desired neighborhood of $\{C_n\}$ in $\mathcal{X}$. So, let $\{D_n\}\in V$ with $\{D_n\}\xrightarrow[H]{} y$. Then we have $$\textrm{d}_{\textrm{H}}(x,y)\leqslant\textrm{d}_{\textrm{H}}(x,D_{n_0})+\textrm{d}_{\textrm{H}}(D_{n_0},y)\leqslant\textrm{d}_{\textrm{H}}(x,C_{n_0})+\textrm{d}_{\textrm{H}}(D_{n_0},y)<2\varepsilon_{n_0},$$ so $y=\varphi(\{D_n\})\in\textrm{B}(x,\varepsilon)\subset U$.

Moreover, the map $\varphi:\mathcal{X}\rightarrow X$ is surjective. For every $x\in X$, we shall construct an element of the inverse limit explicitly. To do so, let $x\in X$ and consider, for every $n\in\mathbb{N}$, the sets $X^n=\textrm{B}(x,\varepsilon_n)\cap A_n$. These sets are finite and non-empty, because, for every $n\in\mathbb{N}$, $A_n$ is a finite  $\varepsilon_n$-approximation. Now we define, for every $n\in\mathbb{N}$, $$X_n^*=\bigcap_{m>n}p_{n,m}(X^m),$$   which are non-empty sets, as an intersection of a nested collection of finite (hence closed) sets in a compact space. To show that it is indeed a nested sequence, we need to prove that, for every $x\in X$ and $n<m$, $p_{n,m+1}(X^{m+1})\subset p_{n,m}(X^m)$. We first show that, for every $m\in\mathbb{N}$, $p_{m,m+1}(X^{m+1})\subset X^m$. Let $d\in p_{m,m+1}(X^{m+1})$. There is an element $c\in X^{m+1}$ such that $d\in p_{m,m+1}(\{c\})$, so $\textrm{d}(x,c)<\varepsilon_{m+1}$ and $\textrm{d}(c,d)<\overline{\varepsilon_{m}}$ and we get $$\textrm{d}(x,d)\leqslant\textrm{d}(x,c)+\textrm{d}(c,d)<\underline{\varepsilon_{m}}+\overline{\varepsilon_m}=\varepsilon_m,$$ meaning that $d\in X^m$. Now, it follows directly that $$p_{n,m+1}(X^{m+1})=p_{n,m}(p_{m,m+1}(X^{m+1}))\subset p_{n,m}(X^m).$$ The sequence $X^*=\{X_n^*\}$ is an element of the inverse limit $\mathcal{X}$. This is so because, for every $n\in\mathbb{N}$, $\textrm{diam}X_n^*<2\varepsilon_n$ (by construction, $X_n^*\subset X^n$) and, for every pair $n<m$, we have $p_{n,m}(X_m^*)=X_n^*$. We just need to prove it for two consecutive terms, i.e., we want to prove that, for every $\todon$, $p_{n,n+1}(X_{n+1}^*)=X_n^*$, and the result follows inductively. The last assertion relies on the following fact\footnote{This is a kind of Mittag-Leffer property for these elements of the inverse limit.}: For every $n\in\mathbb{N}$ there exist an integer $*(n)>n$ such that, for every $m\geqslant *(n)$, $X_n^*=p_{n,m}(X^m)$. The proof goes by construction. For every $z\in X^n\backslash X_n^*$ there exists $n_z\in\mathbb{N}$ such that, for every $m\geqslant n_z$, $z\notin p_{n,m}(X^m)$. Being $X^n$ a finite set, consider $$*(n)=\max\conjunto{n_z:z\in X^n\backslash X_n^*}=\min\conjunto{m\in\mathbb{N}:p_{n,m}(X^m)=X_n^*},$$ and we have the desired result. The function $*:\mathbb{N}\rightarrow\mathbb{N}$ is an increasing function. Considering any $m\geqslant*(n+1)$ is elementary to see that  $$p_{n,n+1}(X_{n+1}^*)=p_{n,n+1}(p_{n+1,m}(X^m))=p_{n,m}(X^{m})=X_n^*,$$ as wanted. We claim that $\varphi(X^*)=x$. For every $\varepsilon>0$, consider $n_0$ such that $\varepsilon_{n_0}<\varepsilon$. Then, for every $n>n_0$, we have that $\textrm{d}_{\textrm{H}}(x,X_n^*)<\varepsilon_n<\varepsilon$, because, for every $x^*\in X_n^*$, $\textrm{d}(x,x^*)<\varepsilon_n$, and then, $X_n^*\subset\textrm{B}(x,\varepsilon_n)$ and $x\in\textrm{B}(X_n^*,\varepsilon_n)$.


The proof of the surjectivity gives us an important element of the inverse limit related with each $x\in X$. By construction, this element of the inverse limit is maximal in the following sense: For every $\{C_n\}\in\mathcal{X}$, such that $x=\varphi(\{C_n\})$, we have that $C_n\subset X_n^*$, for every $n\in\mathbb{N}$. Indeed, for every $m\in\mathbb{N}$, $\textrm{d}_{\textrm{H}}(x,C_m)<\varepsilon_m$ so $C_m\subset\textrm{B}(x,\varepsilon_m)\subset X^m$. Now, given $\todon$, for every $m>*(n)$, $C_n=p_{n,m}(C_m)\subset p_{n,m}(X^m)=X_n^*$. Actually, we can alternatively define $X_n^*$ just with this property as $$X_n^*=\bigcup_{\{C_n\}\in\varphi^{-1}(x)}C_n,$$ because of the maximal property and that $\varphi(\{X_n^*\})=x$.

The previous construction allows us to define a map on the other direction, $\phi:X\rightarrow\mathcal{X}$ with $\phi(x)=\{X_n^*\}$. To prove that this map is continuous in every point, let us consider a neighborhood $V$ of $\{X_n^*\}$ in $\mathcal{X}$. We know that there exists a neighborhood of the form $$W=\left(2^{X_1^*}\times 2^{X_2^*}\times\ldots\times 2^{X_r^*}\times U_{2\varepsilon_{r+1}}(A_{r+1})\times\ldots\right)\cap\mathcal{X}$$ such that $W\subset V$. We need to find points close enought to $x$, that is, an open neighborhood $U\subset X$ such that $\phi(U)\subset W$. We do this by the following construction. First of all, consider $s=*(r)$. We use the following notation, not to be confused with the usual topological notation:
\begin{align*}
\overline{X^s}&:=\overline{\textrm{B}}(x,\varepsilon_s)\cap A_s\enspace(\textrm{where}\enspace\overline{\textrm{B}}(x,\varepsilon_s)=\conjunto{y\in X:\textrm{d}(x,y)\leqslant\varepsilon_s}),\\
\partial X^s&:=\left(\overline{\textrm{B}}(x,\varepsilon_s)\setminus\textrm{B}(x,\varepsilon_s)\right)\cap A_s,\\
X_{\delta}^s&:=\textrm{B}(x,\varepsilon_s+\delta),\enspace\textrm{for}\enspace\delta\in(-\varepsilon_s,\infty).
\end{align*}
Let us consider the distance from $x$ to the closest point of $A_s$ that is not in $\overline{X^s}$, $$\varepsilon_s^+(x)=\min\conjunto{\textrm{d}(x,a):a\in A_s\setminus\overline{X^s}}=\textrm{d}(x,A_s\setminus\overline{X^s})>\varepsilon_s.$$ If there is not such a point, the proof is easier, just consider $\varepsilon_s^+(x)=2\varepsilon_s$. 
In general, we claim the following (see figure \ref{stretch}): 
\begin{figure}[h]
\begin{center}
\includegraphics[scale=1]{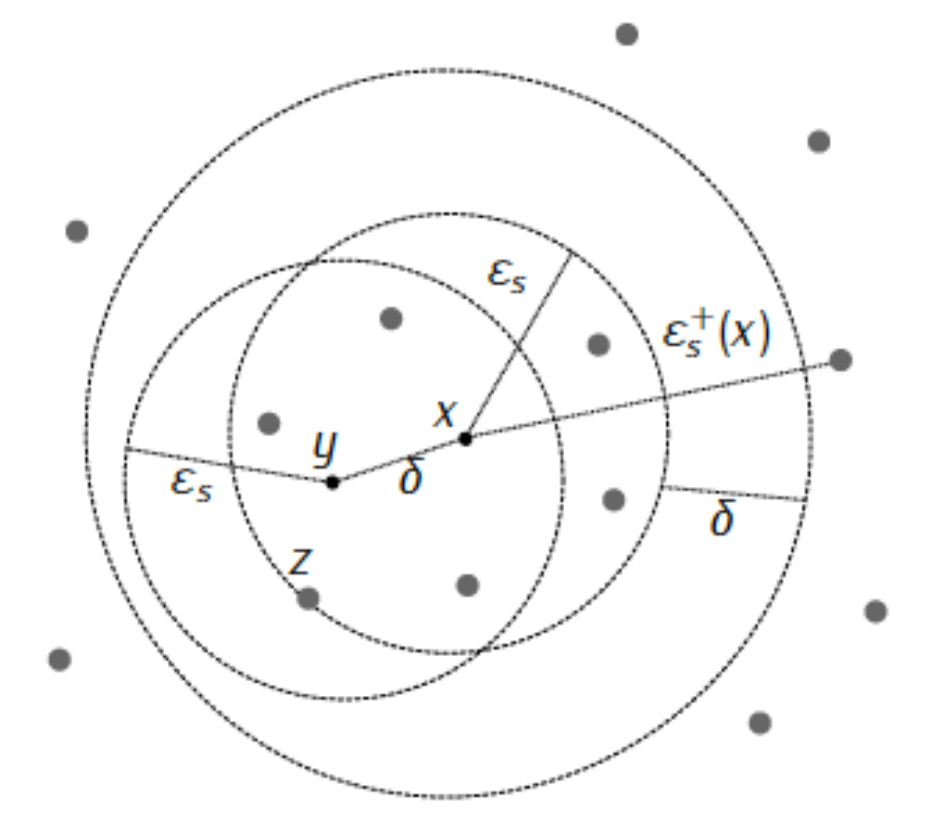}
\caption{For points $y$, close enough to $x$, we do not add exterior points of $X^s$, when we consider its $\varepsilon_s$-neighborhoods. Possibly, some points of the boundary $z\in\partial X^s$ are included, but they are not the image of any point in the next step.}
\label{stretch}
\end{center}
\end{figure}
\begin{itemize}
\item [$i)$] For every $\delta<\varepsilon_s^+(x)-\varepsilon_s$, $\overline{X^s}=X^s_{\delta}$: If $c\in X_{\delta}^s$ then $\textrm{d}(x,c)<\varepsilon_s+\delta<\varepsilon_s^+(x)$, so $c\in\overline{X^s}$.
\item [$ii)$] For every $\delta<\underline{\varepsilon_s}$ we have that, for every $y\in\textrm{B}(x,\delta)$, $p_{s,s+1}(Y^{s+1})\subset Y^s\setminus\partial X^s$: Consider $z\in\partial X^s$ and $b\in Y^{s+1}$. Then, $$\varepsilon_s=\textrm{d}(x,z)\leqslant\textrm{d}(x,y)+\textrm{d}(y,b)+\textrm{d}(b,z)<2\underline{\varepsilon_s}+\textrm{d}(b,z),$$
so $\textrm{d}(b,z)>\gamma_s$, and then $z\notin p_{s,s+1}(Y_{s+1})$. That means $p_{s,s+1}(Y^{s+1})\cap\partial X^s=\emptyset$, hence $p_{s,s+1}(Y^{s+1})\subset Y^s\setminus\partial X^s$.
\end{itemize}
The desired neighborhood of $x$ is $U=\textrm{B}(x,\delta)$ with $$\delta<\min\conjunto{\varepsilon_s^+(x)-\varepsilon_s,\underline{\varepsilon_s}}.$$

For $y\in\textrm{B}(x,\delta)$, we have that $Y^s\subset X^s_{\delta}=\overline{X^s}$ and that $$Y^*_r\subset p_{r,s+1}(Y^{s+1})=p_{r,s}(p_{s,s+1}(Y^{s+1}))\subset p_{r,s}(Y^s\setminus\partial X^s)\subset p_{r,s}(X^s)=X^*_r.$$ For $n<r$, $Y_n^*=p_{n,r}(Y_r^*)\subset p_{n,r}(X_r^*)=X_n^*$. Then $\{Y_n^*\}\subset W$ and hence the map $\phi:X\rightarrow\mathcal{X}$ is continuous. 

This map is clearly injective. Suppose we have two diferent points $x,y$ of $X$. Then, they are at distance, let us say, $\varepsilon=\textrm{d}(x,y)$. Consider $s\in\mathbb{N}$ such that $\varepsilon_s<\frac{\epsilon}{2}$. Then, for every $n>s$, we have that $\textrm{B}(x,\varepsilon_n)\cap\textrm{B}(y,\varepsilon_n)=\emptyset$, that is $X_n\cap Y_n=\emptyset$. So, necesarilly, we have that $X_n^*\cap Y_n^*=\emptyset$, which implies that $\{X^*_n\}\neq \{Y^*_n\}$, being the map injective.

If we consider the restriction to the image $\mathcal{X}^*=\phi(X)$, then $\phi:X\rightarrow\mathcal{X}^*$ is bijective.
But, it is easy to see that $\mathcal{X}^*$ is Hausdorff. If we consider two different points $\{X^*_n\},\{Y^*_n\}\in\mathcal{X}^*$, then there exist $x\neq y$ such that $\{X^*_n\}=\phi(x)$ and $\{Y^*\}=\phi(y)$. Repeating the last proof, we obtain an $s\in\mathbb{N}$ such that, for every $n>s$ we have $X_n^*\cap Y_n^*=\emptyset$. So, we claim that the neighborhoods $$\left(2^{X_1^*}\times\ldots\times2^{X_s^*}\times2^{X_{s+1}^*}\times U_{2\varepsilon_{s+2}}(A_{s+2})\times\ldots\right)\cap\mathcal{X}^*$$ and $$\left(2^{Y_1^*}\times\ldots\times2^{Y_s^*}\times2^{Y_{s+1}^*}\times U_{2\varepsilon_{s+2}}(A_{s+2})\times\ldots\right)\cap\mathcal{X}^*$$ of $\{X_n^*\}$ and $\{Y_n^*\}$ respectively in $\mathcal{X}^*$, are disjoint. Hence $\mathcal{X}^*$ is Hausdorff. Then, as a bijective and continuous map between a compact Hausdorff space and a Hausdorff space, we get that the map  $\phi:X\rightarrow\mathcal{X}^*$ is a homeomorphism.

So, we have that $\mathcal{X}^*$ is a homeomorphic copy of $X$ inside $\mathcal{X}$. Now, we will see that $\mathcal{X}^*$ is a strong deformation retract of $\mathcal{X}$. To do so, we consider the compositions of the maps defined above. It is very easy to see that $\varphi\cdot\phi:X\rightarrow X$ is the identity map. It is not that easy to see that the map $\phi\cdot\varphi:\mathcal{X}\rightarrow\mathcal{X}$ is homotopic to the identity $1_{\mathcal{X}}$. We will write the homotopy explicitly: It is the map $H:\mathcal{X}\times[0,1]\rightarrow\mathcal{X}$ given by
\begin{equation*}
H(\{C_n\},t) = \left\{
\begin{array}{lll}
\{C_n\} & \text{if} & t\in[0,1),\\
\phi\cdot\varphi(\{C_n\}) & \text{if} & t=1.
\end{array}\right.
\end{equation*}
We only need to show the continuity at the points $(\{C_n\},1)\in\mathcal{X}\times[0,1]$. Let us write $\phi\cdot\varphi(\{C_n\})=\phi(x)=\{X^*_n\}$. Consider any neighborhhod $V$ of $\{X^*_n\}$ in $\mathcal{X}$. We can obtain a neighborhood of $\{X^*_n\}$ of the form $$W=\left(2^{X_1^*}\times\ldots\times2^{X^*_{r}}\times U_{2\varepsilon_{r+1}}(A_{r+1})\times\ldots\right)\cap\mathcal{X}$$ such that $W\subset V$. As we have done in a previous proof, we consider $s=*(r)$, $\varepsilon_s^+(x)=\textrm{d}(x,A_s\setminus\overline{X^s})$, and $\delta<\min\conjunto{\varepsilon_s^+(x)-\varepsilon_s,\underline{\varepsilon_s}}$. Select $t>s$ such that $\varepsilon_t<\frac{\delta}{2}$. We claim that the neighborhood $$U=\left(2^{C_1}\times 2^{C_2}\times\ldots\times2^{C_t}\times U_{2\varepsilon_{t+1}}(A_{t+1})\times\ldots\right)\cap\mathcal{X}$$ of $(\{C_n\},1)$ in $\mathcal{X}\times[0,1]$ satisfies $H(U\times[0,1])\subset W$. Let $(\{D_n\},t)\in U\times[0,1]$ where $D_n\subset C_n$ for $n=1,\ldots, t$. Then, if $t<1$, $H(\{D_n\},t)=\{D_n\}\subset W$, because $r<s<t$ and $D_n\subset C_n\subset X_n^*$ for $n=1,\ldots,s$. On the other hand, if $t=1$, then $H(\{D_n\},1)=\phi\cdot\varphi(\{D_n\})=\phi(y)=\{Y^*_n\}$. This implies that $\{Y^*_n\}\in W$. To see why, first observe that, for every $d\in D_t\subset C_t\subset X^*_t$, $\textrm{d}(x,y)\leqslant\textrm{d}(x,d)+\textrm{d}(d,y)<2\varepsilon_t<\delta$. Then, again as before, $Y^t\subset X^t\cup\partial X^t$ and $Y_r^*\subset X_r^*$, so $\{Y^*_n\}\in W$, and the homotopy is then continuous. The space $\mathcal{X}^*$ is a strong deformation retract of $\mathcal{X}$ and the proof of the theorem is finished$\qedhere$
\end{proof}

This theorem leads us easily to the following consequence.
\begin{cor}\label{teo:homotopiafinitos}
For every compact metric space, there exists an inverse sequence of finite spaces whose inverse limit has the same homotopy type.
\end{cor}
This means that the homotopy type of compact metric spaces is, in some sense, pro-finite.
\subsection{Example}\label{ej:intervalotriadico}
Let $X=[0,1]$ be the unit interval with the usual metric on the real line $\textrm{d}$. In order to perform the Main Construction, we are going to use subdivisions of the unit interval in powers of $\frac{1}{3}$. The conditions of the construction will force us to take, for each subdivision, an small enough subdivision for the next step. The diameter of $X$ is $M=1$. Let us select $\varepsilon_1=2>M$ and $A_1=\{0\}$. Obviously $U_{2\varepsilon_1}(A_1)=\{0\}$. Then, it is easy to see that $\gamma_1=1$. 

For the next step, consider $$\varepsilon_2=\frac{1}{3}<\frac{\varepsilon_1-\gamma_1}{2}=\frac{1}{2}$$ and $A_2=\{\frac{k}{3}, k=0,\ldots,3\}$. Then $$U_{2\epsilon_2}(A_2)=A_2\cup\left\lbrace\left\lbrace\frac{k}{3},\frac{k+1}{3}\right\rbrace,k=0,\ldots,2\right\rbrace,$$ because, for $k<k^{\prime}$, we have $\textrm{d}(\frac{k}{3},\frac{k^{\prime}}{3})<\frac{2}{3}$ if and only if $k^{\prime}-k<2$, that is, $k^{\prime}-k$ is either $0$ or $1$. The largest distance of a point of $X$ to a point of the approximation $A_2$ is reached when the point lies exactly in the middle of two consecutive points of $A_2$, hence $\gamma_2=\frac{1}{6}$. 

We pick  $$\varepsilon_3=\frac{1}{3^3}<\frac{\varepsilon_2-\gamma_2}{2},\frac{\delta_2}{2}=\frac{1}{3\cdot2^2}$$ and $$A_3=\left\lbrace\frac{k}{3^3},\enspace k=0,\ldots,3^3\right\rbrace$$ for the new approximation. Then $$U_{2\varepsilon_3}(A_3)=A_3\cup\left\lbrace\left\lbrace\frac{k}{3^3},\frac{k+1}{3^3}\right\rbrace,\enspace k=0,\ldots,3^3-1\right\rbrace,$$ because, for $k<k^{\prime}$ we have $\textrm{d}(\frac{k}{3^3},\frac{k^{\prime}}{3^3})<\frac{2}{3^3}$ if and only if $k-k^{\prime}<2$, that is $k$ and $k^{\prime}$ are the same or consecutive integers. Now, $\gamma_3=\frac{1}{2\cdot3^3}$ (the middle of interval argument holds again). 

For the next approximation, we would select $$\varepsilon_4=\frac{1}{3^5}<\frac{\varepsilon_3-\gamma_3}{2}=\frac{1}{3^3\cdot2^2}.$$

Following this process, we can take, for an arbitrary $n\in\mathbb{N}$, $\varepsilon_n=\frac{1}{3^{2n-3}}$ and $$A_n=\left\lbrace\frac{k}{3^{2n-3}},\enspace k=0,\ldots,3^{2^n-3}\right\rbrace$$ is an $\varepsilon_n$-approximation of $[0,1]$. Observe that $$U_{2\varepsilon_n}(A_n)=A_n\cup\left\lbrace\left\lbrace\frac{k}{3^{2n-3}},\frac{k+1}{3^{2n-3}}\right\rbrace,\enspace k=0,\ldots,3^{2n-3}-1\right\rbrace.$$ We can calculate, as in the previous steps, the numbers required to continue the process. The maximum distance of a point of the unit interval to one of the approximation is reached in the middle of any interval formed by two consecutive points of the approximation, so $\gamma_n=\frac{1}{2\cdot3^{2n-3}}$. 

Next step will consist of taking $$\varepsilon_{n+1}<\frac{1}{2^2\cdot3^{2n-3}}.$$
So we are allowed to choose $\varepsilon_{n+1}=\frac{1}{3^{2n-3+2}}=\frac{1}{3^{2(n+1)-3}}$ and $$A_{n+1}=\left\lbrace\frac{k}{3^{2(n+1)-3}},\enspace k=0,\ldots,3^{2(n+1)-3}\right\rbrace$$ is an $\varepsilon_{n+1}$-approximation of $[0,1]$ and then this construction can be done in this way for every $n\in\mathbb{N}$ obtaining an explicit \textsc{fas} for the space $X$.

Considering this construction done, we analyze some facts about the proof of the main theorem in this example. First, observe that there are points $x\in X$ of the interval with only one point in the preimage by the map $\varphi$, i.e., there is only one point of the inverse limit $\mathcal{X}$ converging to $x$ in the Hausdorff metric. To clarify this, let us see what happens at $x=0\in[0,1]$. It is obvious that the point $\left(0,0,\ldots\right)\in\mathcal{X}$ converges to $0$ in the Hausdorff metric. If we want a different element of the inverse limit converging to $0$, it is natural to think that we could use the fact that $\lim_{n\rightarrow\infty}\frac{1}{3^{2n-3}}=0$ to obtain it, but it turns out that $$\left(0,\frac{1}{3},\frac{1}{3^3},\ldots,\frac{1}{3^{2n-3}},\frac{1}{3^{2(n+1)-3}},\ldots\right)\notin\mathcal{X}$$ because, for every $n\in\mathbb{N}$, $p_{n,n+1}\left(\frac{1}{3^{2(n+1)-3}}\right)=0$. If we try to construct the "maximal" element $X^*$ of the inverse limit with image $x=0$, we obtain that, for every $n\in\mathbb{N}$, $$X_n=\textrm{B}(0,\varepsilon_n)\cap A_n=\{0\}$$ and $p_{n,m}(0)=0$, for every $n<m$. Hence, $$X_n^*=\bigcap_{n<m}p_{n,m}(X_m)=0,$$ for every $n\in\mathbb{N}$, and then, $X^*=\left(0,0,\ldots\right)$. Henceforth, every element of $\mathcal{X}$ converging to $0$ should be "contained" in this one, in the way we explained before. But there is no possibility appart from $X^*$.

Actually, every point in $A_n$ for some $n\in\mathbb{N}$ has this property. For any of them, let us say $x=\frac{k}{3^{2n-1}}$, we have $$X_n=\textrm{B}\left(\frac{k}{3^{2n-1}},\frac{1}{3^{2n-1}}\right)\cap A_n=\left\lbrace\frac{k}{3^{2n-1}}\right\rbrace$$ and $p_{n,m}(x)=x$ for $n<m$ and $p_{n,n-1}(x)=0$.  So, the only element of $\mathcal{X}$ converging to $x$ is $X^*=\left(0,\ldots,0,\frac{k}{3^{2n-1}},\frac{k}{3^{2n-1}},\ldots\right)$. The subset of the interval consisting of these points, $\bigcup_{n\in\mathbb{N}} A_n$, is dense in the unit interval. 


If we choose a point not in this subset, for example $x=\frac{1}{2}$, we obtain a different result. First of all, we know that $\frac{1}{2}$ is not going to be a point of the approximation, for any $n\in\mathbb{N}$, because if that was true, then $\frac{1}{2}=\frac{k}{3^{2n-3}}$ and that implies $3^{2n-3}=2k$ which is impossible. Now we claim that, for every $n\in\mathbb{N}$, $\frac{1}{2}$ is at the same distance of two consecutive points of the approximation and, because of that, both minimize its distance to the approximation. This is true because $$\frac{k+1}{3^{2n-3}}-\frac{1}{2}=\frac{1}{2}-\frac{k}{3^{2n-3}}\Longleftrightarrow k=\frac{3^{2n-3}-1}{2}.$$ Now, let us construct the "maximal" element for this point. We get:
\begin{eqnarray}
X_1&=&0,\nonumber\\
X_2&=&\textrm{B}\left(\frac{1}{2},\frac{1}{3}\right)\cap A_2=\left\lbrace\frac{1}{3},\frac{2}{3}\right\rbrace,\nonumber\\
\ldots\nonumber\\
X_n&=&\textrm{B}\left(\frac{1}{2},\frac{1}{3^{2n-3}}\right)\cap A_n=\left\lbrace\frac{\frac{3^{2n-3}-1}{2}}{3^{2n-3}},\frac{\frac{3^{2n-3}+1}{2}}{3^{2n-3}}\right\rbrace,\nonumber\\
\ldots\nonumber
\end{eqnarray}
It is easy to see that $p_{n,m}(X_m)=X_n$ for every $n<m$ so $X_n^*=X_n$ for every $n\in\mathbb{N}$. So, for $x=\frac{1}{2}$, $$X^*=\left(0,\left\lbrace\frac{1}{3},\frac{2}{3}\right\rbrace,\left\lbrace\frac{14}{3^3},\frac{15}{3^3}\right\rbrace,\ldots,\left\lbrace\frac{\frac{3^{2n-3}-1}{2}}{3^{2n-3}},\frac{\frac{3^{2n-3}+1}{2}}{3^{2n-3}}\right\rbrace,\ldots\right)$$ which obviously converges to $\frac{1}{2}$ with the Hausdorff metric.
But now, we can see that each term has two elements and the maps $p_{n,m}$ are sending the first to the first and the second to the second, so we can consider the sequences
\begin{eqnarray}
C_1&=&\left(0,\left\lbrace\frac{1}{3}\right\rbrace,\left\lbrace\frac{14}{3^3}\right\rbrace,\ldots,\left\lbrace\frac{\frac{3^{2n-3}-1}{2}}{3^{2n-3}}\right\rbrace,\ldots\right)\nonumber\\
C_2&=&\left(0,\left\lbrace\frac{2}{3}\right\rbrace,\left\lbrace\frac{15}{3^3}\right\rbrace,\ldots,\left\lbrace\frac{\frac{3^{2n-3}+1}{2}}{3^{2n-3}}\right\rbrace,\ldots\right)\nonumber
\end{eqnarray}
and claim that $C_1,\enspace C_2\in\mathcal{X}$ and both converge to $\frac{1}{2}$. So this point has exactly three points of the inverse limit in its inverse image by $\varphi$, being that map not injective.

\subsection{Some special cases}
The same property about the injectivity of $\varphi$ in some points observed in the previous example holds in general.
\begin{prop}\label{teo:inyectiva}
Let $X$ be a compact metric space and $\{\varepsilon_n,A_n,\gamma_n,\delta_n\}_{\todon}$ a \textsc{fas} of $X$. If a point $x\in X$ satisfies $x\in A_n$, for every $n\geqslant n_0$, for some $n_0\in\mathbb{N}$, then the cardinality of $\varphi^{-1}(x)$ is one.
\end{prop}
\begin{proof}
We are going to prove that, if $x\in X$ belongs to $A_n$ for every $n\geqslant n_0$, then $$X^*=\left(p_{1,n_0}(x),\ldots,p_{n_0-1,n_0}(x),x,x,\ldots\right).$$ So, there are no more points $C\in\mathcal{X}$ satisfying $\varphi(C)=x$, appart from $X^*$ (because of the maximality of $X^*$). For $n\geqslant n_0$, we have that $x\in X_n=\textrm{B}(x,\varepsilon_n)\cap A_n$, and then, $x\in X_n^*$ for every $n\geqslant n_0$, because $p_{n,m}(x)=x$ for every $n_0\leqslant n<m$. So $X^*$ has the form $$X^*=\left(p_{1,n_0}(X^*_{n_0}),\ldots,p_{n_0-1,n_0}(X^*_{n_0}),X^*_{n_0},X^*_{n_0+1},\ldots\right).$$ Now we prove that $X^*_n=\{x\}$ for every $n\geqslant n_0$. Let $y_0\in X_{n_0}$. Then, $y\in X^*_{n_0}$ if and only if there exists, for every $i\in\mathbb{N}$, $y_i\in A_{n_0+i}$ such that $y\in p_{n_0+i,n_0+i+1}(y_{i+1})$ and $y_i\in X_{n_0+1}$ for every $i\in\mathbb{N}$. We are going to see that, if there is a chain of points satisfiying the first condition, they cannot satisfy the second. So, let us suppose there exists a chain $y_i\in A_{n_0+i}$, for every $i\in\mathbb{N}$ such that one belongs to the image of the following. For the sake of simplicity, let us write $d_i:=\textrm{d}(x,y_i)$ for $i\in\mathbb{N}$, (and $d_0=\textrm{d}(x,y_0)$). For every $i\in\mathbb{N}$, $y_{i+1}$ is closer (or at the same distance) to $y_i$ than to $x$, so we have $$d_{i+1}\geqslant\textrm{d}(y_i,y_{i+1})<\gamma_{n_0+i}.$$ Moreover, it is obvious that $d_i\leqslant d_{i+1}+\textrm{d}(y_{i+1},y_i)$, i.e., $d_i-d_{i+1}\leqslant\textrm{d}(y_{i+1},y_i)$. Combining this with the previous observation, we get $d_i-d_{i+1}<\gamma_{n_0+1}$. On the other hand, we have that, for every $i\in\mathbb{N}$, $d_i\leqslant d_{i+1}+\textrm{d}(y_{i+1},y_i)\leqslant 2d_{i+1}$, so $d_{i+m}\geqslant\frac{d_i}{2^m}$. We supposed $y_0\in X_{n_0}$, so $\varepsilon_{n_0}-d_0>0$. We claim that, for every $i\in\mathbb{N}$, $$\varepsilon_{n_0+i}-d_i<\frac{2\varepsilon_{n_0}-(i+2)d_0}{2^{i+1}}.$$ We prove it by induction. The first case is
\begin{eqnarray*}
\varepsilon_{n_0+1}-d_1 & < & \frac{\varepsilon_{n_0}-\gamma_{n_0}}{2}-d_1<\frac{\varepsilon_{n_0}-d_0}{2}-\frac{d_1}{2}\\
& \leqslant & \frac{\varepsilon_{n_0}-d_0}{2}-\frac{d_0}{2^2}=\frac{2\varepsilon_{n_0}-3d_0}{2^2}.
\end{eqnarray*}
Now, suppose the hypothesis of induction is satisfied, and we check
\begin{eqnarray*}
\varepsilon_{n_0+i+1}-d_{i+1} & < & \frac{\varepsilon_{n_0+i}-\gamma_{n_0+i}}{2}-d_{i+1}<\frac{\varepsilon_{n_0+i}-d_{i}}{2}-\frac{d_{i+1}}{2}\\
& < & \frac{2\varepsilon_{n_0}-(i+2)d_0}{2^{i+2}}-\frac{d_0}{2^{i+2}}=\frac{2\varepsilon_{n_0}-(i+3)d_0}{2^{i+2}}.
\end{eqnarray*}
It is obvious that there exists an $i\in\mathbb{N}$ such that $(i+3)d_0>2\varepsilon_{n_0}$. For this $i$, we have that $\varepsilon_{n_0}-d_i<0$, so $y_i\notin X_{n_0+i}$, and then, $y_0\notin X^*_{n_0}$. We conclude $X^*_{n_0}=\{x\}$ and the same argument can be applied to show that $X_n^*=\{x\}$, for every $n\geqslant n_0\qedhere$
\end{proof}

In view of this result, it is natural to look for \textsc{fas} making the map $\varphi$ the more "injective" posible, i.e., injective in the largest possible set of points.
\begin{obs} 
For every \textsc{fas} of $X$, the set $\bigcup_{\todon}A_n$ is always dense in $X$. For each open set $U\subset X$ there exists $x\in U$ and $\varepsilon>0$ such that $\textrm{B}(x,\epsilon)\subset U$. Let us select $n_0\in\mathbb{N}$ such that $\textrm{B}(x,\epsilon_{n_0})\subset\textrm{B}(x,\epsilon)$. Then, for any $a\in A_{n_0}$ with $\textrm{d}(x,a)<\epsilon_{n_0}$, we have that $\textrm{B}(x,\epsilon)\cap A_{n_0}\neq\emptyset$. This also shows that every compact metric space has a countable dense subset.
\end{obs}
We want to apply Proposition \ref{teo:inyectiva} to obtain inyectivity in a dense subset of $X$. The following construction will be useful.
\begin{cons}
For every compact metric space $X$, there exists a \textsc{fas} with $A_n\subset A_{n+1}$ for every $\todon$: If $M=\textrm{diam}(X)$, let us consider $\epsilon_1>M$, and $A_1=\{x\}$ with $x\in X$. Then, consider $\epsilon_2$ with the usual rule. Now, for $A_2$, we take the union $A_2=A'_2\cap A_1$ where $A'_2$ is a $\epsilon_2$-approximation of $X$, then so is $A_2$. We can proceed in this way for every $\todon$. If we have that $A_n$ is a $\epsilon_n$-approximation of $X$, consider $\gamma_n$ and take $\epsilon_{n+1}$ as always. Then consider $A_{n+1}$ as the union $A'_{n+1}\cap A_n$ where $A'_{n+1}$ is a $\epsilon_{n+1}$-approximation of $X$ and, hence, $A_{n+1}$ too. In this way, we obtain a \textsc{fas} of $X$ such that $A_n\subset A_{n+1}$ for every $\todon$.
\end{cons}
The best situations we can have are the following.
\paragraph{Countable spaces}
\begin{prop}
Let $X$ be a countable metric space. There exists a \textsc{fas} of $X$ with inverse limit $\mathcal{X}$ homeomorphic to $X$.
\end{prop}
\begin{proof}
We can write $X=\{x_1,x_2,\ldots,x_n,\dots\}$. We just need to find a \textsc{fas} for $X$ satisfying $x_n\in A_n$ and $A_n\subset A_{n+1}$ for every $\todon$. The first condition gives us $\bigcup_{\todon}A_n=X$ and the second one will make $\varphi$ injective on the set $$\varphi^{-1}\left(\bigcup_{\todon}A_n\right)=\varphi^{-1}(X)=\mathcal{X},$$ and then, $\varphi:\mathcal{X}\rightarrow X$ will be a homeomorphism. There are many ways of doing so. We can just do the general construction forcing each $A_n$ to contain $x_n$ and $A_m$, for every $m<n$. For example, if we consider, for every $\todon$, the numbers $$r(n)=\min\big{\lbrace} i\in\mathbb{N}:\{x_1,\ldots,x_i\} \enspace\textrm{is a}\enspace\epsilon_n\enspace\textrm{approximation of}\enspace X\big{\rbrace},$$ it is clear that $r(n+1)\geqslant r(n)$ and then we can write the approximations as
\begin{eqnarray}
A_1&=&\{x_1\},\nonumber\\
A_2&=&\{x_1,\ldots,x_{r(2)}\},\nonumber\\
&\ldots&\nonumber\\
A_n&=&\{x_1,\ldots,x_{r(n)}\},\nonumber\\
&\ldots&\nonumber
\end{eqnarray}
and we are done$\qedhere$
\end{proof}
Now we face the case of proper dense subsets of $X$. First of all, we observe the following
\begin{obs}
For every dense subset $Y\subset X$ of a compact metric space, and every $\epsilon>0$, there exists an $\epsilon$-approximation $A\subset Y$: Let us consider the covering $\lbrace\ball{x}{\frac{\varepsilon}{2}}:x\in X\rbrace$ and a finite subcovering $\lbrace\ball{x_1}{\frac{\epsilon}{2}},\ldots,\ball{x_k}{\frac{\epsilon}{2}}\rbrace$. Now we take $y_1,\ldots,y_k\in Y$ such that $\dist{x_i}{y_i}<\frac{\epsilon}{2}$ for every $i=1,\ldots,k$, so $\{y_1,\ldots,y_k\}$ is an $\epsilon$-approximation of $X$.
\end{obs}
We can state  the main result in this direction
\begin{prop}
For every countable dense subset of a compact metric space, $Y\subset X$, there exists a \textsc{fas} of $X$ such that there is a dense subset of $\mathcal{X}$ which is homeomorphic to $Y$. 
\end{prop}
\begin{proof}
If we write $Y=\{y_1,y_2,\ldots,y_n,\ldots\}$, it is easy to obtain a \textsc{fas} of $X$ such that $A_n\subset A_{n+1}$, for every $\todon$, and $\bigcup_{\todon}A_n=Y$. For example, using the previous remark, we can take as approximations 
\begin{eqnarray}
A_1&=&\{y_1\}\nonumber\\
A_2&=&\{y_2\}\cup A_1\cup A'_2\textrm{ with }A'_2\subset Y\textrm{ an }\epsilon_2\textrm{-approximation of }X,\nonumber\\
A_3&=&\{y_3\}\cup A_2\cup A'_3\textrm{ with }A'_3\subset Y\textrm{ an }\epsilon_3\textrm{-approximation of }X,\nonumber\\
&\ldots&\nonumber\\
A_n&=&\{y_n\}\cup A_{n-1}\cup A'_n\textrm{ with }A'_n\subset Y\textrm{ an }\epsilon_n\textrm{-approximation of }X,\nonumber\\
&\ldots&\nonumber
\end{eqnarray}
If we restrict the map $\varphi:\mathcal{X}\rightarrow X\supset Y=\bigcup_{\todon}A_n$ to the set $\varphi^{-1}(Y)$ we obtain that $$\varphi\mid_{\varphi^{-1}(Y)}:\varphi^{-1}(Y)\longrightarrow Y$$ is injective and hence a homeomorphism. So $\varphi^{-1}$ is the desired set. We have the inclusions $\varphi^{-1}(Y)\subset\mathcal{X}^*\subset\mathcal{X}$, by construction (recall Proposition \ref{teo:inyectiva}). Now, to see that $\varphi^{-1}(Y)$ is dense in $\mathcal{X}^*$. Let $V$ be any open set of $\mathcal{X}^*$ and $C\in V$ any point of it, where $C=(C_1,C_2,\ldots,C_n,\ldots)$. Choose an open neighborhood from the basis $$C\in W=(2^{C_1}\times\ldots\times2^{C_m}\times U_{2\epsilon_{m+1}}(A_{m+1})\times\ldots)\cap\mathcal{X}\subset V$$ and select any $c\in C_m$. Then, $c^*=(\ldots,c,c,\ldots)$, because $A_n\subset A_{n+1}$ for every $\todon$. So $c^*\in W\cap\varphi^{-1}(Y)\subset V\cap\varphi^{-1}(Y)$, which implies $\overline{\varphi^{-1}(Y)}=\mathcal{X}^*\qedhere$
\end{proof}
\begin{obs}
The inclusion $\varphi^{-1}(Y)$ of last proposition is proper: Recall example \ref{ej:intervalotriadico} where $Y=\bigcup_{\todon}A_n$ with $$A_n=\left\lbrace\frac{k}{3^{2n-3}}:k=0,1,\ldots,3^{2n-3}\right\rbrace$$ and, while $\frac{1}{2}^*$ is obviously an element of $\mathcal{X}^*$, it does not belong to $\varphi^{-1}(Y)$, since $\frac{1}{2}$ does not belong to any approximation $A_n$.
\end{obs}
\paragraph{Ultrametric spaces}
Another structures that will give us a complete topological reconstruction are ultrametric spaces. An ultrametric space $X$ is a metric space with an extra property of the distance: Instead of satisfying just the triangle inequality, they satisfy the strong triangle inequality, which is:
$$\forall x,y,z\in X,\enspace\dist{x}{y}\leqslant\max\left\lbrace\dist{x}{z},\dist{y}{z}\right\rbrace.$$
This inequality gives us some properties that make the ultrametric spaces very special ones. For example\footnote{See chapter 2 of \cite{Ra} for more properties and detailed proofs about ultrametric spaces.}:
\begin{enumerate}
\item[$\cdot$]Every triangle is isosceles, with the non equal side smaller than the other two.
\item[$\cdot$]For every $x,y\in X$ and $\epsilon\geqslant\delta>0$, $\ball{x}{\epsilon}\cap\ball{y}{\delta}\neq\emptyset$ implies that $\ball{y}{\delta}\subset\ball{x}{\epsilon}$.
\end{enumerate}

We want to show that, for the case of ultrametric spaces, there exists a \textsc{fas} such that they recover the topological type of the space. The key idea here, is that for those spaces there are very special approximations:
\begin{lem}Let $X$ be a compact ultrametric space. For every $\epsilon>0$, there exists an $\epsilon$-approximation of $X$, $\{x_1,\ldots,x_k\}$, such that $\ball{x_i}{\epsilon}\cap\ball{x_j}{\epsilon}=\emptyset$ for every $i\neq j$.
\end{lem}
\begin{proof}
The covering by open balls $\{\ball{x}{\epsilon}:x\in X\}$ of $X$ has a finite subcover $\{\ball{x_1}{\epsilon},\ldots,\ball{x_k}\}$. So, $\{x_1,\ldots,x_k\}$ is an $\epsilon$-approximation of $X$. Now for any $i\neq j$ such that $\ball{x_i}{\epsilon}\cap\ball{x_j}{\epsilon}\neq\emptyset$ it turns out that $\ball{x_i}{\epsilon}=\ball{x_j}{\epsilon}\qedhere$
\end{proof}
We can state the main theorem about ultrametric spaces.
\begin{teo}
For every compact ultrametric space $X$, there exists a \textsc{fas} with inverse limit $\mathcal{X}=X$.
\end{teo}
\begin{proof}
Let us consider any \textsc{fas} $\fas$ of $X$ satisfying the property stated in the previous lemma. Then, for every $x\in X$ and every $\todon$, we have that $\card(q_{A_n}(x))=1$: Let us suppose that $a_1,a_2\in q_{A_n}(x)$. Then, $\dist{x}{a_1},\dist{x}{a_2}<\gamma_n<\epsilon_n$ but, in that case, we will have that $x\in\ball{a_1}{\epsilon_n}\cap\ball{a_2}{\epsilon_n}$ which is not possible. Then, $q_{A_n}:X\rightarrow A_n$ is actually a single valued continuous map. Moreover, if we restrict to $A_{n+1}$, we obtain that $$q_{A_n}\mid_{A_{n+1}}=p_{n,n+1}\mid_{A_{n+1}}:A_{n+1}\longrightarrow A_n$$ is a continuous map. So, it makes sense to write the following diagram, 
$$\xymatrix @C=15mm{X\ar[dr]^{q_{A_n}} \ar[d]_{q_{A_{n+1}}} & \\
A_{n+1}\ar[r]_{p_{n+1,n}} & A_n,}$$ which, moreover, is commutative. If not, then there would exist $a_1,a_2\in A_n$ with $q_{A_n}(x)=a_1$ and $p_{n,n+1}q_{A_{n+1}}(x)=a_2$. Clearly, $\dist{x}{a_1}<\epsilon_n$, but also 
\begin{eqnarray*}\dist{x}{a_2} & \leqslant & \dist{x}{q_{A_{n+1}}(x)}+\dist{q_{A_{n+1}}(x)}{p_{n,n+1}q_{A_{n+1}}(x)}<\\ & < & \gamma_{n+1}+\gamma_n<\epsilon_{n+1}+\gamma_n<\frac{\varepsilon_n-\gamma_n}{2}+\gamma_n<\epsilon_n.
\end{eqnarray*} 
and this is imposible, since then $x\in\ball{a_1}{\epsilon_n}\cap\ball{a_2}{\epsilon_n}$. Adding that $q_{A_n}$ is always a surjective map distinguishing points of $X$ (see corollary 3 on page 61 of \cite{MSshape}), we have that $X$ is the inverse limit $X=\lim_{\leftarrow}(A_n,p_{n,n+1})$. Now, it remains to see that every element of the inverse limit $C=(C_1,C_2,\ldots,C_n,\ldots)\in\mathcal{X}$ satisfies that $\card({C_n})=1$ for every $\todon$. If not, for any pair $a_1,a_2\in C_n$ we would have that $\dist{x}{a_1},\dist{x}{a_2}<\epsilon$, with $x=\lim_H\{C_n\}$, which, again, is not possible. So, we have that $$\mathcal{X}=\lim_{\leftarrow}\left(U_{2\epsilon_n}(A_n),p_{n,n+1}\right)=\lim_{\leftarrow}(A_n,p_{n,n+1})=X\qedhere$$
\end{proof}
\subsection{Previous work on inverse limits of finite spaces}
Our construction is a sequence of finite spaces, which, in the limit, are able to reflect topological properties of the original space. Its main features are: 
\begin{enumerate}
\item[$\cdot$] It is internal: It is constructed from the space itself, without need of external ambient spaces to approximate them. We use the hyperspace, which is constructed just in terms of the compac metric space.
\item[$\cdot$] It is constructive: Given a space explicitly, we can actually select points for each approximation. This is important, since it allows us to perform explictly the construction over the space, and possibly determine some topological structure, up to some error.  
\end{enumerate}
There exist previous results on the approximation of topological spaces by finite spaces. This is an old theme, but nowadays it is becoming more important because of its role in the emerging field of computational topology. In this section we will review some of these results and compare them with ours.

\paragraph{Approximation of compact polyhedra}
There is a paper of E. Clader \cite{Cinverse} where the following theorem is proved:
\begin{teo}[E. Clader]\label{teo:clader}
Every compact polyhedron is homotopy equivalent to the inverse limit of an inverse sequence of finite spaces.
\end{teo}
The proof consists of taking as finite spaces the vertices of the barycentric subdivisions of the simplicial complexes defining the compact polyhedron. Given a simplicial complex, the McCord correspondence assigns a finite $T_0$ space. Clader assigned the opposite topology to these finite spaces. That is, consider for every $\todon$, the $n$-th barycentric subdivision $K^{(n)}$ and the finite space $F_n=\mathcal{X}(K^{n})^{op}$, that is, the $n$-th barycentric subdivision of the finite space $\mathcal{X}(K)$ with the opposite topology of that assigned by McCord. Then, there is a natural map $p_n$ from $|K|$ to each $F_n$, because every point of $|K|$ belongs to a unique simplex of $K^{(n)}$. For every $n>1$, there is a map $q_n:F_n\rightarrow F_{n-1}$ closing the diagram with $p_n$ and $p_{n-1}$. Then, it is shown that the polyhedron is a retract of the inverse limit of these finite spaces and maps. 

Note that every compact polyhedron is a compact metric space (for details of the metric, see, for example, the appendix on polyhedra of \cite{MSshape}). So, this theorem is a particular case of corollary \ref{teo:homotopiafinitos}. 

 
\paragraph{Finite approximations and Hausdorff reflections}
In a series of papers, R. Kopperman \emph{et al.} (\cite{KTWthe}, \cite{KWfinite}) proved the following theorem about finite approximations.
\begin{teo} [R. Kopperman, R. Wilson]\label{teo:koppwil} Every compact Hausdorff space is the Hausdorff reflection of the inverse limit of an inverse system of finite spaces.
\end{teo}
The finite spaces involved in this proof are constructed in a very theoretical way. It is considered the set of all possible open coverings of the space and then the spaces are defined with a boolean algebra on the open sets of the coverings. This theorem has the advantage that it is very general: It works for any compact Hausdorff space, with no need for a metric. But it has the desadvantage that it is not always explicitly constructible and that we loose a lot of intuition with the Hausdorff reflection.

The idea of a reflection of a topological space is to construct another space, as similar as posible to the first one with an extra separation property and a universal map. Somehow it is similar to compactification. Concretely, given a topological space $X$ and a separation property $T$, we will say that $\mu_X:X\rightarrow X_T$ is the $T$ \emph{reflection} of $X$ if
$\mu_X$ is a continuous map, $X_T$ has the property $T$ and every continuous map $f:X\rightarrow Z$ with $Z$ having property $T$, factors through a map $g:X_T\rightarrow Z$, i.e., the diagram
$$\xymatrix @C=20mm{X \ar[d]_{\mu_X}\ar[r]^{f}& Z \\
X_T\ar[ur]_{g}}$$ commutes. If the map $\mu_X$ is surjective we will say that the reflection is surjective, too. For some properties, the existence of reflectors is a well known fact.
\begin{teo}(see \cite{MNtopics}, chapter 14)Let $X$ be a topological space. For $T$ being the separation properties $T_i,\enspace i=0,1,2,3,3\frac{1}{2}$ there exists a surjective reflection.
\end{teo}
It is easy to see that two reflections of the same space are homeomorphic.

In many cases, reflections are obtained as quotient spaces (not in every case, as for example the Tychonoff functor -or reflection- which is not obtained as a quotient) for a relation. Nevertheless, it is not allways obtained as the obvious relation. As a matter of fact, in order to obtain the Hausdorff reflection we need to define the following relations (see the reference \cite{SBMa}, a short and beautiful paper about reflections, where this is shown):
\begin{itemize}
\item $x R_1 y$ iff for every pair of neighborhoods $U_x, U_y$ of $x, y$ resp., we have $U_x\bigcap U_y\neq\emptyset$.
\item $x R_2 y$ iff there exist $x=z_1,z_2,\ldots,z_n=y$ such that $z_1R_1z_2R_1\ldots R_1z_n$.
\item $x R_3 y$ iff for every $f:X\rightarrow Z$, with $Z$ Hausdorff, we have $f(x)=f(y)$.
\end{itemize}
Then, the Hausdorff reflection of $X$ is the quotient space $X_H=X/R_3$.

We want to compare the Hausdorff reflection of a topological space with the space itself in terms of shape type. As a motivation, we can cite \cite{Mon}, where it is shown that the Tychonoff functor indeed induces the identity morphism in shape. So, a topological space and its Tychonoff reflection have the same shape. We will show the same holds for the Hausdorff reflection. 
\begin{lem}
The Hausdorff reflection of the product $X\times I$, where $I=[0,1]$, is homeomorphic to $X_H\times I$.
\end{lem}
\begin{proof}
Consider the continuous map $$\begin{array}{rcl}f:X\times I&
\longrightarrow &X_H\times I\nonumber\\(x,t)&\longmapsto &
(\mu_X(x),t),\nonumber\end{array}$$ which is a quotient map. Moreover, the space $X_H\times I$ is Hausdorff, so there exists a continuous surjective map $h:(X\times I)_H\rightarrow X_H\times I$
such that the diagram
$$\xymatrix @C=20mm{X\times I\ar[d]_{\mu_{X\times I}}\ar[r]^{f}& X_H\times I \\
(X\times I)_H\ar[ur]_{h}}$$ commutes. We see that $h$ is actually a homeomorphism. First of all, $h$ is a quotient map, because $f$ and $\mu_{X\times
I}$ are (\cite{Egeneral}, pag 91). Also, it is an injective map. Indeed, let
$[a],[b]\in(X\times I)_H$ such that $h([a])=h([b])=([z],t)$. Considering that
$\mu_{X\times I}$ is surjective, there exist $(x,t_1),(x,t_2)$ such that
$\mu_{X\times I}(x,t_1)=[a]$ and $\mu_{X\times I}(y,t_2)=[b]$. Because of the commutativity of the previous diagram we have that
\begin{eqnarray}
([x],t_1)=f(x,t_1)=h(\mu_{X\times
I}(x,t_1))=h([a])=([z],t)\nonumber\\
([y],t_2)=f(y,t_2)=h(\mu_{X\times I}(y,t_2))=h([b])=([z],t)\nonumber
\end{eqnarray}
so $[x]=[y]=[z]$ and $t_1=t_2=t$. For this concrete $t$,
we consider the commutative diagram $$\xymatrix @C=20mm{X\ar[r]^{id\times t}\ar[d]_{\mu_X}& X\times I\ar[d]^{\mu_{X\times I}} \\
X_H\ar[r]_g & (X\times I)_H},$$ which exists for being $\mu_{X\times
I}\circ(id\times t):X\rightarrow(X\times I)_H$ a continuous map to a Hausdorff space. We consider the images of $x,y$ by the two different maps of the diagram.
As $[x]=[y]$, we obtain that
$[a]=[b]$, so $h$ is injective. A quotient and injective map is a homeomorphism$\qedhere$
\end{proof} 
\begin{teo}
For every topological space $X$, the Hausdorff reflection $\mu_X:X\rightarrow X_H$ induces the identity morphism in shape.
\end{teo}
\begin{proof}

To show this, we are going to use the characterization of identity morphisms in shape\footnote{See \cite{Mshapes} for this result of shape theory}:The map $\mu_X:X\rightarrow X_H$ is the identity morphism in shape if and only if the map 
$$\begin{array}{rcl}[X_H,P]&
\longrightarrow &[X,P]\\
h&\longmapsto & h\cdot f,\\ \end{array}$$ with $P$ being any metric ANR, is bijective.

It is surjective: Given a map $g:X\rightarrow P$,
with $P$ ANR and then, Hausdorff, there exists a map
$h:X_H\rightarrow P$ such that $g=h\cdot\mu_X$, that is, what we wanted.
It is injective: Let $h_1,h_2:X_H\rightarrow P$, with $P$ ANR, two continuous maps such that $h_1\cdot\mu_X$ y $h_2\cdot\mu_X$
are homotopic, i.e., there exists a continuous map, $G:X\times
I\rightarrow P$ such that $G(x,0)=h_1\cdot\mu_X(x)$ and
$G(x,1)=h_2\cdot\mu_X(x)$. Being $P$ Hausdorff, there exists a continuous map $F:(X\times I)_H\rightarrow P$ such that
$G=F\cdot\mu_{X\times I}$. Applying the previous lemma, we get
$\mu_{X\times I}=\mu_X\times id$, so we have that the following diagram commutes:
$$\xymatrix @C=20mm{X\times I\ar[r]^G\ar[d]_{\mu_X\times id}& P\\
X_H\times I\ar[ur]_F & }.$$ Then, for every $x\in X$, we have
\begin{eqnarray}
F([x],0)=G(x,0)=h_1\cdot\mu_X(x)=h_1([x])\nonumber\\
F([x],1)=G(x,1)=h_2\cdot\mu_X(x)=h_2([x]).\nonumber
\end{eqnarray}
So, $h_1$ and $h_2$ are homotopic$\qedhere$
\end{proof} 
\begin{cor}
A topological space $X$ has the same shape than its Hausdorff reflection $X_H$.
\end{cor}

Note that with Theorem \ref{teo:koppwil} and the result just proved here about the shape of the Hausdorff reflection we will get the following generalization of Theorem \ref{teo:koppwil}.
\begin{cor}
Every compact Hausdorff space has the same shape as the inverse limit of an inverse system of finite spaces.
\end{cor}

In an attempt of understanding better the Hausdorff reflection of an inverse system of spaces, Kopperman and Wilson proved that the original space is not only the Hausdorff reflection but the set of closed points of the inverse limit. We can prove the same in our construction.
\begin{prop}
For every compact metric space $X$ and every \textsc{fas} of $X$, the space $\mathcal{X}^*$ is just the set of closed points of $\mathcal{X}$. Moreover it is its Hausdorff reflection $\mathcal{X}^*=X_H$.
\end{prop}
\begin{proof}
First of all, we are going to characterize, for every $x\in X$ the point of the inverse limit $X^*=\phi(x)$. It is the set $$X^*=\bigcap_{C\in\varphi^{-1}(x)}\overline{\{C\}}.$$ We divide the proof:
\begin{enumerate}
\item[$(\subset)$]We show here that if $\varphi(X^*)=\varphi(C)=x$ (notation: $X^*=(X_1^*,X_2^*,\ldots)$), then $X^*\in\overline{\{C\}}$. Let $X^*\in V$ an open neighborhood in $\mathcal{X}$. Then, there exists an open neighborhood $$X^*\in U=(2^{X_1^*}\times 2^{X_2^*}\times\ldots\times 2^{X_r^*}\times U_{2\epsilon_{r+1}}(A_{r+1})\times\ldots)\cap\mathcal{X}.$$ But, obviously, $C\in U$, so $C\in U\cap\{C\}\neq\emptyset$.
\item[$(\supset)$]Let $D=(D_1,D_2,\ldots)\in\bigcap_{C\in\varphi^{-1}(x)}\overline{\{C\}}$.Then $\{D_n\}$ converges to $x$ in the Hausdorf metric. So, for every $D\in U$ open neighborhood, we have that $U\cap\{X^*\}\neq\emptyset$. In particular, for every $r\in\mathbb{N}$ we have neighborhoods of the form $$(2^{D_1}\times 2^{D_2}\times\ldots\times 2^{D_r}\times U_{2\epsilon_{r+1}}(A_{r+1})\times\ldots)\cap\mathcal{X},$$ where $X^*$ belongs. So, for every $r\in\mathbb{N}$ we have $X_r^*=D_r$, hence $X^*=D$.
\end{enumerate}

Now to show that $\mathcal{X}^*$ is the set of closed points, first observe that every $X^*\in\mathcal{X}^*$ is $X^*=\bigcap_{C\in\varphi^{-1}(x)}\overline{\{C\}}$, so a closed set. On the other hand, if there is a closed point $C\in\mathcal{X}$, with $\varphi(C)=y$ then $Y^*\in\overline{\{C\}}=\{C\}$ so $C=Y^*\in\mathcal{X}^*.$

To show that $\mathcal{X}^*$ is actually the Hausdorff reflection of $\mathcal{X}$, let us consider a continuous map $\alpha:\mathcal{X}\rightarrow Y$ with $Y$a Hausdorff space. Consider two points $C,C'\in\mathcal{X}$ such that $\varphi(C)=\varphi(C')=x=\varphi(X^*)$, with $X^*\in\mathcal{X}^*$. Then, using the previous characterization of $X^*$, we have that $X^*\in\overline{\{C\}}\cap\overline{\{C'\}}$. Then, applying the map $\alpha$, and using that it is continuous and that $Y$ is Hausdorff, we obtain
\begin{eqnarray}
 \alpha(X^*)&\in&\alpha(\overline{\{C\}})\cap\alpha(\overline{\{C'\}})\subset\nonumber\\
 &\subset&\overline{\{\alpha(C)\}}\cap\overline{\{\alpha(C')\}}=\nonumber\\
 &=&\{\alpha(C)\}\cap\{\alpha(C')\},\nonumber
\end{eqnarray}
so, $\alpha(X^*)=\alpha(C)=\alpha(C')$. Now, we claim that the map $\phi\cdot\varphi:\mathcal{X}\rightarrow\mathcal{X}^*$ is actually the Hausdorff reflection of $\mathcal{X}$. This is so, because the map $\alpha\mid_{\mathcal{X}^*}:\mathcal{X}^*\rightarrow Y$ makes the diagram
$$\xymatrix @C=15mm{\mathcal{X}\ar[r]^{\phi\cdot\varphi} \ar[d]_{\alpha} & \mathcal{X}^*\ar[dl]^{\alpha\mid_{\mathcal{X}^*}}\\
Y &}$$
commutative and $\alpha\mid_{\mathcal{X}^*}$ is continuous since $\phi\cdot\varphi$ is a retraction and hence a identification$\qedhere$
\end{proof}
\section{Inverse Persistence}
It is clear that a persistence module is nothing but an inverse sequence of vector spaces and homomorphisms reversed, in the sense that the sequence grows in the opposite direction. Moreover, if we "cut" the inverse sequence at some step, we obtain a persistence module of finite type and, hence, the corresponding barcode. In this way, we can obtain persistence modules from inverse sequences of spaces and this makes a connection between shape theory and persistent homology. In this section, we present a different way of persistence in which the persistence modules are obtained from inverse sequences of polyhedra. 

Let us consider a compact metric space $X$ and a polyhedral approximative sequence obtained by the main construction as in section \ref{mainconstruction} $\conjunto{K_n,p_{nn+1}}$. Although all these inverse sequences of polyhedra are the realizations of inverse sequences of simplicial complexes and simplicial maps between them, they are not filtrations of simplicial complexes, even obviating the finiteness condition, since the maps involved are not the inclusion. But, if we consider, for any $p\in\mathbb{N}$, and a field $F$, the induced homology inverse sequences $\conjunto{H_p(K_n,p_{nn+1};F)}$ are persistence modules (with maps not induced by the inclusion) of simplicial complexes, but by means of proximity, as indicated in the quoted section.

As a compact metric space, we can perform this construction to a point cloud $\mathbb{X}$. If we consider that this point cloud is a sample, possibly with some noise, of a compact metric space $X$, the results obtained in section \ref{homotopicalreconstruction} and in \cite{Mon} about the homotopical and shape properties recovered in the inverse limit of any \textsc{fas} and \textsc{pas} of $X$ make us think that the topological properties obtained applying this method in $\mathcal{X}$ will represent topological properties of the space $X$. 

Consider a \textsc{fas}, $\fas$ of $\mathbb{X}$. Since $\mathbb{X}$ is finite, $\fas$ has only a finite number of different approximations: There is an integer $s$ such that, for every $n\geqslant s$,
$$2\varepsilon_n<\max\conjunto{\dist{x}{y}:x,y\in X},$$ and hence $A_n=A_{n+1}=X$, $U_{2\varepsilon_n}(A_n)=U_{2\varepsilon_{n+1}}(A_{n+1})$ and $p_{nn+1}=id$. So, we have only a finite number $s$ of "changes" in the sequence, that we can be written as 
$$U_{\varepsilon_1}(A_1)\xleftarrow{\enspace p_{12}\enspace}U_{\varepsilon_2}(A_2)\xleftarrow{}\ldots\xleftarrow{}U_{\varepsilon_{s-1}}(A_{s-1})\xleftarrow{\enspace p_{s-1s}\enspace}U_{\varepsilon_s}(A_s).$$
Now, consider the induced polyhedral approximative sequence of section \ref{mainconstruction}, 
$$K_1\xleftarrow{\enspace p_{12}\enspace}K_2\xleftarrow{}\ldots\xleftarrow{} K_{s-1s}\xleftarrow{\enspace p_{s-1s}\enspace} K_s.$$ Its induced $p$-th singular homology sequence (for a field $F$)
$$H_p(K_1)\xleftarrow{\enspace p_{12}\enspace}H_p(K_2)\xleftarrow{}\ldots\xleftarrow{} H_p(K_{s-1s})\xleftarrow{\enspace p_{s-1s}\enspace} H_p(K_s)$$
is indeed a persistence module of finite type, so it has an associated barcode $\mathcal{B}_{\mathbb{X}}$ that we will call \emph{inverse barcode}. We call this procedure Inverse Persistence. In Figure \ref{inversepersistence} we represent the Inverse Persistence process. We list some differences between Persistence and Inverse Persistence.
\begin{enumerate}
\item The simplicial complexes used in regular persistence are constructed using all the set of points of the point cloud for every level. In contrast, the simplicial complexes constructed in the inverse persistence are based on subsets of the point cloud. Moreover, we need to add more points to the finite spaces, in order to make the maps between them continuous.
\item The maps used in the finite sequence of polyhedra constructed from the point cloud are always inclusions in regular persistence, but they are not in inverse persistence. Although they are not inclusions, they are consistent is some sense because they are defined in terms of proximity and, as we have seen, they are constructed in a way that, carried until the infinity, captures the homotopical and shape properties of the space.
\end{enumerate}

For the analysis of the inverse persistence we propose the following steps.
\begin{enumerate}
\item Formalize the algorithm outlined here and compare the computational cost with the standard algorithms for persistence on point clouds.
\item Compare the obtained  inverse persistence modules and compare them with the regular persistence modules in terms of the concept of \emph{interleaving}, introduced in \cite{Cproximity} by Chazal et al.
\item Compare the obtained barcodes from inverse persistence with the ones obtained by regular persistence using the \emph{bottleneck} distance on barcodes (see \cite{Estability} for definition and main results concerning this distance).
\end{enumerate}

It is expected that the inverse persistence modules have the same behaviour as regular ones in terms of stability (see \cite{Estability, Cproximity}), because of the shape theoretical framework where they are constructed. We hope this shape approach to persistence to be suitable for real world applications because of its constructibility and its good properties concerning stability.
\begin{figure}
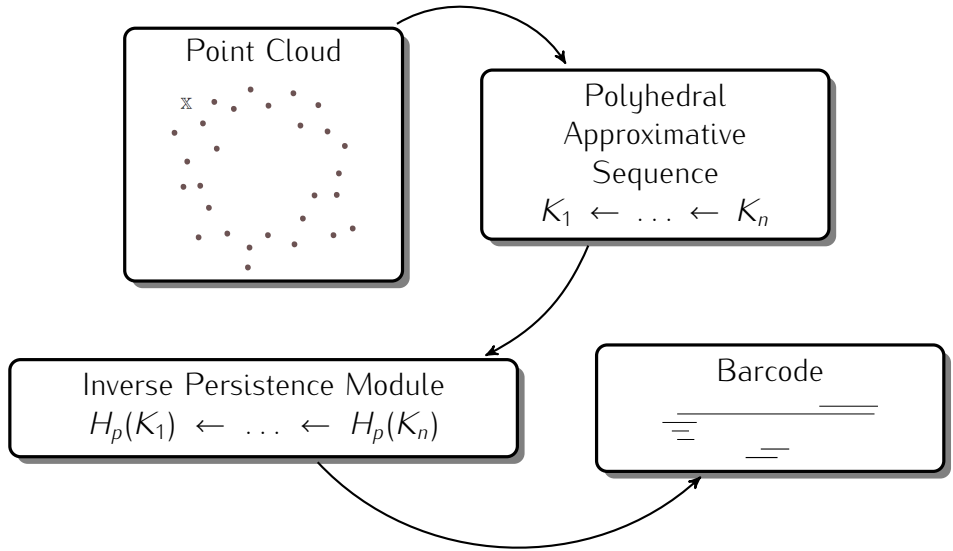
\label{inversepersistence}
\begin{center}
\begin{tikzpicture}[node distance=1cm, auto,]
 \node[punkt] (pointcloud) {Point Cloud\\ \vspace{0.3cm}\includegraphics[scale=0.25]{pers3.pdf}};
 \node[punkta, inner sep=5pt,right=1cm of pointcloud](filtration) {Polyhedral\\ Approximative\\ Sequence\\$K_1\leftarrow\ldots\leftarrow K_n$};
 \node[punktaa, inner sep=5pt,below=1cm of pointcloud](module) {Inverse Persistence Module\\$H_p(K_1)\leftarrow\ldots\leftarrow H_p(K_n)$};
 \node[punktaaa, inner sep=5pt,right=1cm of module](barcode) {Barcode\\ \vspace{0.3cm}\includegraphics[scale=0.40]{barcode.pdf}};
  \draw [pil,bend left=45](pointcloud) edge (filtration);
  \draw [pil,bend left=20](filtration) edge (module);
  \draw [pil,bend right=45](module) edge (barcode);
\end{tikzpicture}
\end{center}
\caption{Process of Inverse Persistence}
\end{figure}

Moreover, we can assign inverse barcodes to every compact space $X$ and every pair of integers $(n,m)$ with $n<m$ using finite parts of the polyhedral approximative sequences
$$H_p(K_n)\longleftarrow\ldots\leftarrow H_p(K_m)$$
We expect inverse barcodes comming from compact metric spaces and (possibly noisy) samples of them to be quite similar in some yet non defined sense. We compute an inverse barcode in the following section.
\subsection{Example: The computational Warsaw Circle}
\label{sec:computationalwarsawcircle}
In this section we will perform the main construction on the Warsaw circle in order to apply the theory previously developed. The Warsaw circle is the paradigmatic example of shape theory and we shall see how inverse persistence works in this case, capturing the shape properties (such as the \v{C}ech homology groups) of this space. See figure \ref{warsawdef}. For computational purposes, we are going to define and work with the following homeomorphic copy of the Warsaw circle. Consider, in $\mathbb{R}^2$, the following segments\footnote{The notation for the segments is $(a,b)-(c,d)$, meaning the segment joining these two points.}: 
\begin{eqnarray}
a_n=\left(\frac{1}{2^{2n-2}},1\right)-\left(\frac{1}{2^{2n-1}},1\right),\nonumber\\
b_{n1}=\left(\frac{1}{2^{2n-2}},\frac{1}{2}\right)-\left(\frac{1}{2^{2n-2}},1\right),\nonumber\\
b_{n2}=\left(\frac{1}{2^{2n-1}},1\right)-\left(\frac{1}{2^{2n-1}},\frac{1}{2}\right),\nonumber\\
c_n=\left(\frac{1}{2^{2n-1}},\frac{1}{2}\right)-\left(\frac{1}{2^{2n}},\frac{1}{2}\right).\nonumber
\end{eqnarray}
Then, the computational Warsaw circle is 
$$\mathcal{W}=(1,0)-(0,0)-(1,0)-(1,1)\bigcup_{\todon}a_n\bigcup_{n\in\mathbb{N}\backslash\{1\}}b_{n1}\bigcup_{\todon}b_{n2}\bigcup_{\todon}c_n.$$
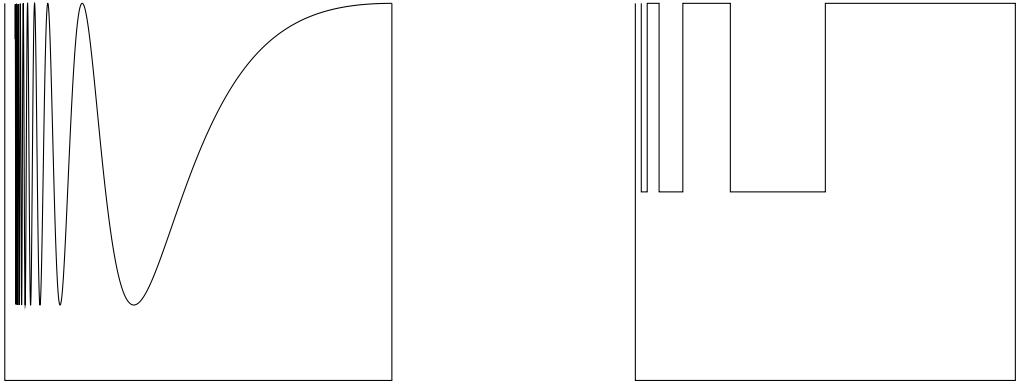
\begin{figure}[h!]
\begin{minipage}[c]{0.45\textwidth}
\begin{center}
\begin{tikzpicture}[scale=2,domain=0:1,x=4cm]
\FPeval{\w}{2/(3.1415)}
  \draw[domain=0.017:\w,samples=5000] plot (\x, {sin((1/\x)r)});
  \draw (0,1)--(0,-1.5)--(\w,-1.5)--(\w,{sin((1/\w)r)});
\end{tikzpicture}
\end{center}
\end{minipage}
\   \ \hfill
\begin{minipage}[c]{0.45\textwidth}
\begin{center}
\begin{tikzpicture}[scale=5,domain=0:1]
\draw (0,1)--(0,0)--(1,0)--(1,1)--(1/2,1)--(1/2,1/2)--(1/4,1/2)--(1/4,1)--(1/8,1)--
(1/8,1/2)--(1/16,1/2)--(1/16,1)--(1/32,1)--(1/32,1/2)--(1/64,1/2)--(1/64,1);
\end{tikzpicture}
\end{center}
\end{minipage}
\caption{The Warsaw circle and the computational Warsaw circle.}
\label{warsawdef}
\end{figure}
We now perform the general construction on $\mathcal{W}$. The diameter of $\mathcal{W}$ is $M=\sqrt{2}$. Then, we can select $\epsilon_1=2\sqrt{2}>M$, and $A_1=\{(0,0)\}$, so $\gamma_1=\sqrt{2}$. In the second step, we take $\epsilon_2=\frac{\sqrt{2}}{2^3}<\frac{\epsilon_1-\gamma_1}{2}=\frac{\sqrt{2}}{2}$. To get an $\epsilon_2$-approximation of $\mathcal{W}$, we explain the process better than giving just the coordinates of the points. Consider the intersection of a grid of side $\frac{1}{2^{3-1}}$, $G_2=\left\lbrace(\frac{l}{2^{3-1}},\frac{m}{2^{3-1}})\in\mathbb{R}^2: l,m\in\mathbb{Z}\right\rbrace$  with $\mathcal{W}$. See figure \ref{2app}. 
\begin{figure}[h!]
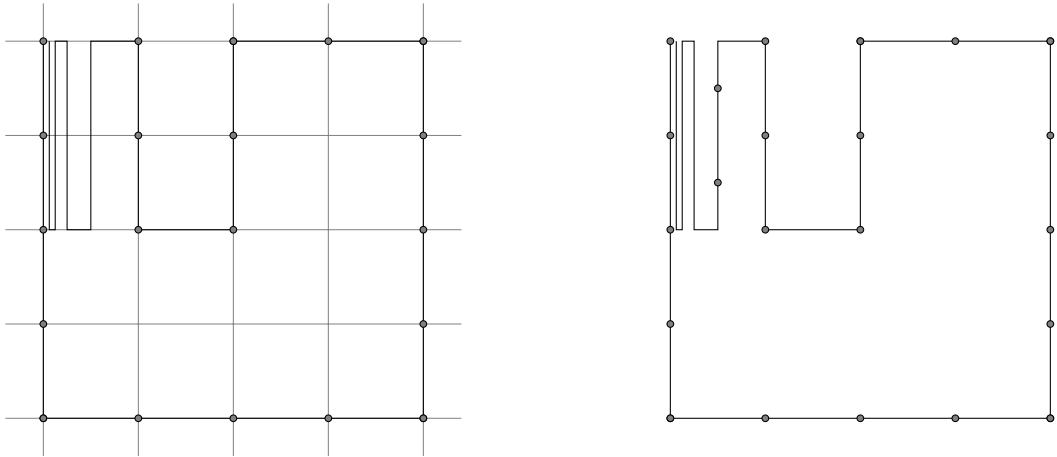

\begin{minipage}[c]{0.45\textwidth}
\begin{center}
\warsawApp{3}{1}
\end{center}
\end{minipage}
\   \ \hfill
\begin{minipage}[c]{0.45\textwidth}
\begin{center}
\warsawApp{3}{0}
\end{center}
\end{minipage}
\caption{The intersection of the grid $G_2$ with $\mathcal{W}$ and the $\epsilon_2$-approximation of $\mathcal{W}$.}
\label{2app}
\end{figure}
Every point of $\mathcal{W}$, not in the upper left square of the grid, and the one just below it, are at distance less or equal to $\frac{1}{2^3}<\epsilon_2$. Concerning the two mentioned squares, we see that every point of $\mathcal{W}$ inside them are at distance less than $\epsilon_2$, except the two centers of the squares, which are exactly at this distance. So, we add these two points and then have an $\epsilon_2$ approximation of $\mathcal{W}$, $$A_2=\left(G_2\cap\mathcal{W}\right)\cup\left\lbrace\left(\frac{1}{2^3},1-\frac{1}{2^3}\right),\left(\frac{1}{2^3},1-\frac{3}{2^3}\right)\right\rbrace.$$
From the picture, we can easily see that $\gamma_2=\frac{1}{2^3}$. We pick\footnote{We want some regularity on the epsilon approximations. All of them will be of the form $\frac{\sqrt{2}}{2^k}$. In this case, there is no $k$ lower than 6 in the inequality. This will be proven for the general case, later.} $\epsilon_3=\frac{\sqrt{2}}{2^6}<\frac{\epsilon_2-\gamma_2}{2}=\frac{\sqrt{2}-1}{2^4}$. To obtain an $\epsilon_3$-approximation of $\mathcal{W}$, we proceed as before. Consider the grid of side $\frac{1}{2^{6-1}}$, $G_3=\left\lbrace(\frac{l}{2^{6-1}},\frac{m}{2^{6-1}})\in\mathbb{R}^2: l,m\in\mathbb{Z}\right\rbrace$ and its intersection with $\mathcal{W}$. Then add the centers of the upper left square of the grid, and the $15=2^4-1$ below it (16 points of $\mathcal{W}$ in total), to obtain an $\epsilon_3$ approximation of $\mathcal{W}$ (see figure \ref{3app}),
$$A_3=\left(G_3\cap\mathcal{W}\right)\cup\left\lbrace\left(\frac{1}{2^6},1-\frac{1}{2^6}\right), \left(\frac{1}{2^6},1-\frac{3}{2^6}\right),\ldots,\left(\frac{1}{2^6},1-\frac{31}{2^6}\right)\right\rbrace.$$
\begin{figure}[h!]
\begin{minipage}[c]{0.45\textwidth}
\begin{center}
\warsawApp{6}{1}
\end{center}
\end{minipage}
\   \ \hfill
\begin{minipage}[c]{0.45\textwidth}
\begin{center}
\warsawApp{6}{0}
\end{center}
\end{minipage}
\caption{The intersection of the grid $G_3$ with $\mathcal{W}$ and the $\epsilon_3$ approximation of $\mathcal{W}$.}
\label{3app}
\end{figure}
Now, it is again clear from the picture, that $\gamma_3=\frac{1}{2^6}$ so we can continue this process to the infinity in the same way. In general, let $\epsilon_n=\frac{\sqrt{2}}{2^{3n-3}}$. Consider the grid of side $\frac{1}{2^{3n-4}}$, $G_n=\left\lbrace\left(\frac{l}{2^{3n-4}},\frac{m}{2^{3n-4}}\right)\in\mathbb{R}^2: l,m\in\mathbb{Z}\right\rbrace$. Then, its intersection with $\mathcal{W}$ and the following $2^{3n-5}$ points, form an $\epsilon_n$ approximation, 
$$A_n=\left(G_n\cap\mathcal{W}\right)\cup\left\lbrace\left(\frac{1}{2^{3n-3}},1-\frac{2k-1}{2^{3n-3}}\right): k=1,\ldots,2^{3n-5}\right\rbrace.$$
It is clear that, again, $\gamma_n=\frac{1}{2^{3n-3}}$ and $\delta_n=\frac{\sqrt{2}}{2^{3n-3}}$. So, writing\footnote{The term $3n-3$ relates the exponent of the denominator with the subindex of each $\epsilon$. We use the $m$ notation for a moment to understand how the denominator is increased in each step without perturbations of another notations.} $m=3n-3$, we need $$\epsilon_{n+1}<\frac{\epsilon_n-\gamma_n}{2}=\frac{\sqrt{2}-1}{2^{m+1}}.$$ We want $\epsilon_{n+1}$ to be of the form $\frac{\sqrt{2}}{2^k}$, so we are looking for $k\in\mathbb{N}$, such that, $\frac{\sqrt{2}}{2^k}<\frac{\sqrt{2}-1}{2^{m+1}}$, i.e., $2^{k-(m+1)}>2+\sqrt{2}$. We can estimate $2^{k-(m+1)}>2+\sqrt{2}>2$ so $k>m+2$. But, actually, $k=m+2$ does not satisfy the first inequality, so we can take any $k\geqslant m+3$, and hence, we choose $\epsilon_{n+1}=\frac{\sqrt{2}}{2^{m+3}}=\frac{\sqrt{2}}{2^{3n}}=\frac{\sqrt{2}}{2^{3(n+1)-3}}$. It is clear, that we can consider an $\epsilon_{n+1}$-approximation as before, intersecting the grid of side $\frac{1}{2^{3(n+1)-4}}=\frac{1}{2^{3n-1}}$, $G_{n+1}=\left\lbrace\left(\frac{l}{2^{3n-1}},\frac{m}{2^{3n-1}}\right)\in\mathbb{R}^2: l,m\in\mathbb{Z}\right\rbrace$ with $\mathcal{W}$ and add $2^{3(n+1)-5}=2^{3n-2}$ points:
$$A_{n+1}=\left(G_{n+1}\cap\mathcal{W}\right)\cup\left\lbrace\left(\frac{1}{2^{3n}},1-\frac{2k-1}{2^{3n}}\right): k=1,\ldots,2^{3n-2}\right\rbrace.$$ Then $\gamma_{n+1}=\frac{1}{2^{3(n+1)-3}}=\frac{1}{2^{3n}}$ and the process is proved to work by induction.

Now, we focus on the Alexandrov-McCord sequence related to this finite approximative sequence. The finite space $A_1$ is just a point, so its associated simplicial complex is just a vertex. In the second step, we have a more interesting case. In figure \ref{3Pol},
\begin{figure}[h!]
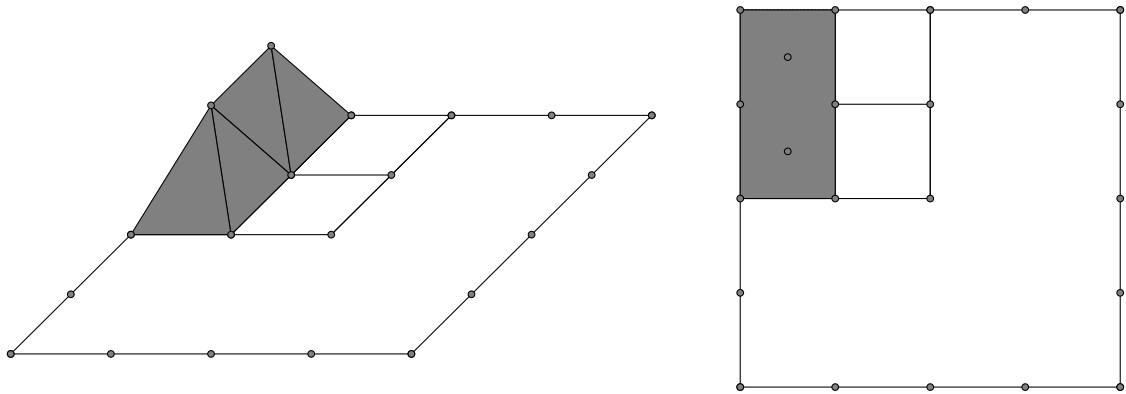

\begin{minipage}[c]{0.38\textwidth}
\begin{center}
\warsawTdPol{3}{30}
\end{center}
\end{minipage}
\   \ \hfill
\begin{minipage}[c]{0.38\textwidth}
\begin{center}
\warsawPol{3}{5}
\end{center}
\end{minipage}
\caption{The realization of the simplicial complex $\mathcal{V}_{2\epsilon_2}(A_2)$ in two perspectives: Lateral and Aerial.}
\label{3Pol}
\end{figure}
we have depicted the polyhedron $\mathcal{V}_{2\epsilon_2}(A_2)$ in two different perspectives. The barycentric subdivision of this polyhedron is exactly the realization of the simplical complex $\mathcal{K}(U_{2\epsilon_2}(A_2))=\mathcal{V}'_{2\epsilon_2}(A_2)$. Actually, the vertices that are not depicted but belong to the subdivision are the points of the space $U_{2\epsilon_2}\backslash A_2$. The 1-simplices of this polyhedron are clear from the picture. But there is more structure. First of all there are two empty squares. At their left, there are two piramids whose cusps represent the points added to the intersection of the grid and $\mathcal{W}$. Between the two piramids, there is a tetrahedron sharing one face with each one of them. The four points of the tetrahedron are the two points added and the two points in common of the two squares (the base of each piramid), which, in the approximation, have diameter less than $2\epsilon_2$, so this tetrahedron is \comillas{filled}. We have to point out that the piramids are empty, that is, their four faces are simplices that are in the polyhedron, but there is no \comillas{solid} base. For the third step, we also depicted the polyhedron $\mathcal{V}_{2\epsilon_3}(A_3)$ (figure \ref{6Pol}), whose barycentric subdivision is $\mathcal{K}(U_{2\epsilon_3}(A_3))=\mathcal{V}'_{2\epsilon_3}(A_3)$. 
\begin{figure}[h!]
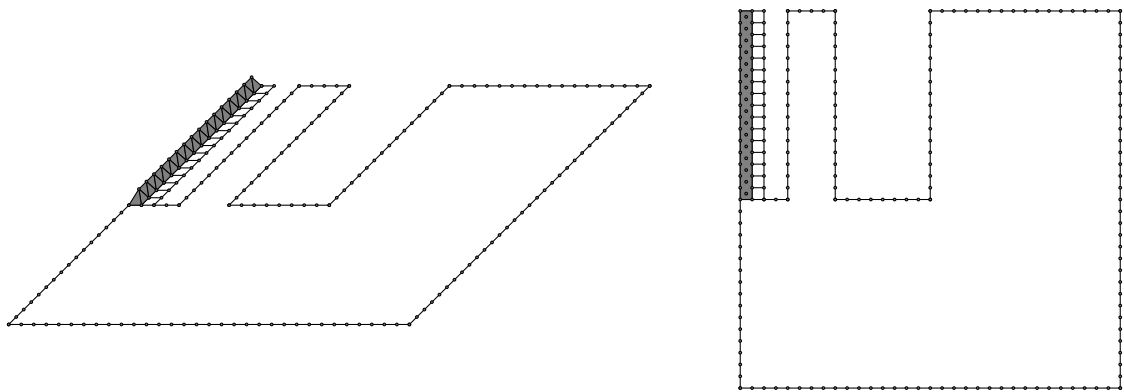

\begin{minipage}[c]{0.38\textwidth}
\begin{center}
\warsawTdPol{6}{30}
\end{center}
\end{minipage}
\   \ \hfill
\begin{minipage}[c]{0.38\textwidth}
\begin{center}
\warsawPol{6}{5}
\end{center}
\end{minipage}
\caption{The realization of the simplicial complex $\mathcal{V}_{2\epsilon_3}(A_3)$ in two perspectives: Lateral and Aerial.}
\label{6Pol}
\end{figure}
The structure of this polyhedron is the same as the previous one. The diference is that it has more 1-simplices, more empty squares ($2^4$) more piramids ($2^4$) and more tetrahedrons ($2^4-1$). In general, for any $\epsilon_n$ approximation we will have the same structure, with $2^{3n-5}$ squares and piramids and $2^{3n-5}-1$ tetrahedrons. Concerning the maps, we can use pictures to see where they send the points of the approximations, and the sets of those points, but we will focus our attention on the induced maps in homology which actually will tell us the behavior of the maps.  

We now study the previous sequence at the homological level. We will compute the first homology group with coefficients in $\mathbb{Z}$ (with notation $H_1(K):=H_1(K;\mathbb{Z})$) of each polyhedron of the sequence and how the induced homology maps work. For the first approximation, everything is trivial. For the second one, we know that $\mathcal{V}_{2\epsilon_2}(A_2)$ (in the figure) has the same homotopy type as $\mathcal{K}(U_{2\epsilon_2}(A_2))$. It has three 1-cycles: The \comillas{big} one and the two little squares. There is no more 1-homology on this complex. This is clear from the aerial perspective in figure \ref{3Pol}. So, the homology group of this polyhedron is just three copies of $\mathbb{Z}$, which we denote $H_1(\mathcal{K}(U_{2\epsilon_2}(A_2)))\simeq\mathbb{Z}^3$. In the third step, as we can see in picture \ref{6Pol}, there is again one \comillas{big} 1-cycle, and $2^4$ small squares. I.e., a total of $2^4+1$ copies of $\mathbb{Z}$, so $H_1(\mathcal{K}(U_{2\epsilon_3}(A_3)))\simeq\mathbb{Z}^{2^4+1}$. We are interested in the map induced in homology by the map $$p_{2,3}:\mathcal{K}(U_{2\epsilon_3}(A_3))\longrightarrow\mathcal{K}(U_{2\epsilon_2}(A_2)).$$ We need to study, for each 1-cycle, where the vertices are sent by the map\footnote{We can visualize the performance of the map by overlying the pictures of the two consecutive approximations, since the map acts in terms of proximity.} An easy reasoning shows that every vertex (included the non drawn ones) in the small 1-cycles of $\mathcal{K}(U_{2\epsilon_3}(A_3))$ are sent to null homologous cycles in $\mathcal{K}(U_{2\epsilon_2}(A_2))$ (let us say, they \comillas{fall} into the shaded part which is the contratible part). However, the big 1-cicle of $\mathcal{K}(U_{2\epsilon_3}(A_3))$ is mapped into the \comillas{big} one of $\mathcal{K}(U_{2\epsilon_2}(A_2))$ (actually, it is mapped into something bigger which retracts into this cycle). So, it is clear, that the induced map in homology, $$(p_{2,3})_*:H_1(\mathcal{K}(U_{2\epsilon_3}(A_3)))\longrightarrow H_1(\mathcal{K}(U_{2\epsilon_2}(A_2))),$$ sends the $2^4$ generators corresponding to the little squares to zero, and the generator of the \comillas{big} 1-cycle to the generator of the \comillas{big} one of the target. So, we get that $\textrm{Im}((p_{2,3})_*)=\mathbb{Z}$. It is readily seen that, if we consider the next step, it will happen the same. In general, the realization of $\mathcal{K}(U_{2\epsilon_n}(A_n))$ has $2^{3n-5}$ 1-cycles corresponding to little squares and one \comillas{big} 1-cycle. So, $H_1(\mathcal{K}(U_{2\epsilon_n}(A_n)))\simeq\mathbb{Z}^{2^{3n-5}+1}$. The map induced by $$p_{n,n+1}:\mathcal{K}(U_{2\epsilon_{n+1}}(A_{n+1}))\longrightarrow\mathcal{K}(U_{2\epsilon_n}(A_n))$$ in homology, $$(p_{n,n+1})_*:H_1(\mathcal{K}(U_{2\epsilon_{n+1}}(A_{n+1})))\longrightarrow H_1(\mathcal{K}(U_{2\epsilon_n}(A_n))),$$ sends the $2^{3n-2}$ 1-cycles corresponding to little squares of $\mathcal{K}(U_{2\epsilon_{n+1}}(A_{n+1}))$ to zero and the 1-cycle corresponding to the \comillas{big} one to the \comillas{big} one in the image $\mathcal{K}(U_{2\epsilon_n}(A_n))$. So, again, the image of the map is $\textrm{im}((p_{n,n+1})_*)=\mathbb{Z}$. So, we see that in each step, the \comillas{big} 1-cycle is the only non-trivial homology that comes from the image of the previous polyhedron. We could say that the \comillas{big} cycle is the only one that survives (or persists) in the whole sequence. In terms of the inverse limit, it is clear that the inverse limit of the inverse sequence induced in homology\footnote{Notation: $\mathcal{K}_n:=\mathcal{K}(U_{2\varepsilon_n}(A_n))$}
$$H_1(\mathcal{K}_1)\xleftarrow{(p_{0,1})_*} H_1(\mathcal{K}_2)\xleftarrow{(p_{1,2})_*}\ldots\xleftarrow{(p_{n-1,n})_*} H_1(\mathcal{K}_n)\xleftarrow{(p_{n,n+1})_*}H_1(\mathcal{K}_{n+1})\xleftarrow{(p_{n+1,n+2})_*}\ldots.$$ is $$\lim_{\leftarrow}\left\lbrace\mathcal{K}_n,(p_{n,n+1})_*\right\rbrace\simeq\mathbb{Z}.$$ Hence, the finite approximations are capturing the \v{C}ech homology of $\mathcal{W}$ in the limit. But, more important, the inverse persistence is able to capture it in a finite number of steps. In this particular example, we consider, for any $n<m$ the steps $n, n+1,\ldots,m$. Then, inverse persistence generate an inverse barcode as in figure \ref{barcode}.
\begin{figure}[h!]
\begin{center}
\begin{tikzpicture}
\draw [line width=0.4mm,  blue] (0,0) -- (10,0)
(0,2/10)--(1/2,2/10)
(0,6/10)--(1/2,6/10)
(0,8/10)--(1/2,8/10)
(1,1+2/10)--(1+1/2,1+2/10)
(1,1+6/10)--(1+1/2,1+6/10)
(1,1+8/10)--(1+1/2,1+8/10)
(10,2+2/10)--(10+1/2,2+2/10)
(10,2+6/10)--(10+1/2,2+6/10)
(10,2+8/10)--(10+1/2,2+8/10);
\draw 
node at (5,2) [] {$\ldots$}
node at (0,-1) [] {$n$}
node at (1,-1) [] {$n+1$}
node at (10,-1) [] {$m$}
node at (0,-1/2) [] {$[$}
node at (1,-1/2) [] {$|$}
node at (10,-1/2) [] {$]$}
node at (0,4/10) [] {{\tiny $\ldots$}}
node at (1,1+4/10) [] {{\tiny $\ldots$}}
node at (10,2+4/10) [] {{\tiny $\ldots$}};
\draw [line width=0.2mm,  black] (0,-1/2)--(10,-1/2);
   \end{tikzpicture}
\caption{The barcode $\mathcal{B}_{\mathcal{W}}$.}
\label{barcode}
\end{center}
\end{figure}
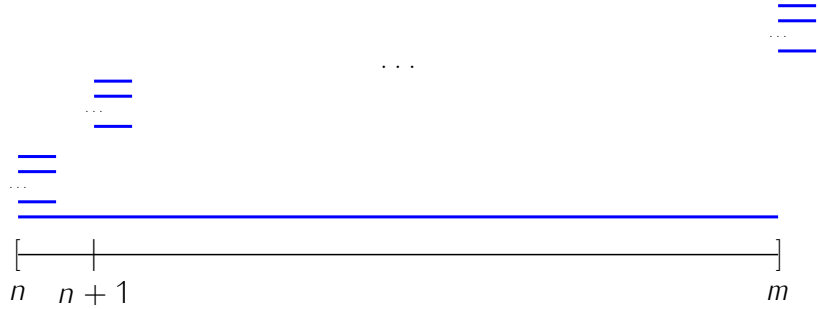
The largest line correspond to the generator of the \comillas{big} hole of $\mathcal{W}$, which is present in every step of the construction and preserved by the maps. Hence, it represents the first \v{C}ech homology group of $\mathcal{W}$. In the other hand, the short lines correspond to the generators of the \comillas{fake} and small homology groups that are created in the construction of $\mathcal{K}_n$ at each step. They are $2^{3i-5}$ generators for each $i=n,n+1,\ldots,m$ as we computed before. The $2^{3(n+1)-5}$ generators of $\mathcal{K}_{n+1}$ are mapped to zero in $\mathcal{K}_n$ and hence they do not generate any long line in the inverse barcode. Whence, the only feature detected by the inverse persistence is the only that actually exists. We hope that experiments carried out in noisy samples of $\mathcal{W}$ could reveal a good behavior of this construction in detecting difficult features of spaces.
\newpage
\addcontentsline{toc}{chapter}{References}
\bibliographystyle{siam}
\bibliography{reconstruction}
\end{document}